\RequirePackage[2022-11-01]{latexrelease}
\ifdefined\reviewversion
\documentclass[review]{siamart171218}
\else
\documentclass{siamart171218}
\fi

\usepackage{amsfonts}
\usepackage{mathtools}

\makeatletter
\@ifundefined{cref@override@label@type}{}{
  \def\cref@override@label@type#1\@nil#2{#1}%
}
\makeatother

\newsiamremark{assumption}{Assumption}
\newsiamremark{remark}{Remark}

\newcommand{\dt}{\Delta t}
\newcommand{\eps}{\varepsilon}
\newcommand{\bfu}{\boldsymbol u}
\newcommand{\bfv}{\boldsymbol v}
\newcommand{\bfw}{\boldsymbol w}

\newcommand{\calE}{\mathcal E}
\newcommand{\calG}{\mathcal G}
\newcommand{\calR}{\mathcal R}
\newcommand{\calN}{\mathcal N}
\newcommand{\calC}{\mathcal C}
\newcommand{\norm}[1]{\left\lVert #1\right\rVert}

\title{An Energy-Stable Implicit Convex-Splitting BDF2 Scheme for the Cahn--Hilliard--Navier--Stokes Equations}
\author{Xuelong Gu\thanks{Department of Mathematics, University of South Carolina, Columbia, SC, 29208, USA.}
\and Qi Wang\thanks{Department of Mathematics, University of South Carolina, Columbia, SC, 29208, USA. Corresponding author. E-mail address: \texttt{QWANG@math.sc.edu}.}}
\headers{An Energy-Stable Implicit CS-BDF2 Scheme for CHNS}{X. Gu and Q. Wang}

\begin{document}
\maketitle

\begin{abstract}
	We develop an energy-stable implicit convex-splitting BDF2 discretization (CS-BDF2) of the Cahn--Hilliard--Navier--Stokes equations.  For the Cahn--Hilliard equation, BDF2 analyses \cite{YanChenWangWise2018,ChenWangYanZhang2019,LiaoJiWangZhang2022} can establish energy stability by testing the phase equation in the \(H^{-1}\) metric.  For CHNS, this test is not compatible with the coupled energy estimate: the momentum equation is tested by \(\bfu^{n+1}\), while the transported phase equation is tested by \(\mu^{n+1}\) so that transport cancels capillary work.  The chemical-potential relation must then be paired with the BDF2 phase increment \((3\phi^{n+1}-4\phi^n+\phi^{n-1})/2\); its nonlinear part must produce a BDF2 bulk-energy difference, up to nonnegative higher-order history terms.  To overcome this difficulty, we introduce a new BDF2-compatible convex-splitting approximation of the nonlinear bulk force that directly yields a discrete bulk-energy identity and enables a discrete energy analysis for the CHNS system. Specifically, we discretize the bulk force \(f(\phi)=\phi^3-\phi\) by
	\[
		\chi(\phi^{\dagger,n+1},\phi^{\dagger,n})-\phi^{*,n+1}, \quad \chi(a,b)=\tfrac14(a^2+b^2)(a+b),
	\]
	where
	\[
		\phi^{\dagger,n+1}=\tfrac{3\phi^{n+1}-\phi^n}{2}, \quad \phi^{\dagger,n}=\tfrac{3\phi^n-\phi^{n-1}}{2}, \quad \phi^{*,n+1}=2\phi^n-\phi^{n-1}.
	\]
	This discretization is based on the shifted BDF2 identity \((3\phi^{n+1}-4\phi^n+\phi^{n-1})/2=\phi^{\dagger,n+1}-\phi^{\dagger,n}\).  With a matching discretization of the reversible coupling terms in CHNS, the scheme is mass conservative, uniquely solvable, and unconditionally energy stable.  We prove second-order convergence for the phase variable, chemical potential, velocity, and pressure.
\end{abstract}



\section{Introduction}

The Cahn--Hilliard--Navier--Stokes (CHNS) equations are a basic
diffuse-interface model for two-phase incompressible flows. They are used to
describe interfacial motion, coarsening, topology changes, and related binary
fluid phenomena without explicitly tracking a sharp interface
\cite{AndersonMcFaddenWheeler1998,KimKangLowengrub2003,Kim2012}. A central
feature of the model is its energy-dissipation structure: kinetic, interfacial,
and bulk energies are coupled through a thermodynamically consistent law. It is
therefore important for numerical discretizations to preserve this structure,
especially in long-time simulations where artificial energy production may
distort interface dynamics. This consideration has motivated many
structure-preserving and energy-stable schemes for phase-field and
hydrodynamic phase-field models, including convex-splitting, stabilized,
decoupled, and auxiliary-variable approaches
\cite{Eyre1998,ShenYang2010,Feng-2013,EQ1,ShenYang2015,HanWang2015,DiegelWangWangWise2017,Gong-2016-Binary,Gong-2018-CHNS,Chen-2020-CHNS,Zhao-2021-CHNS,LiShen2020,LiShen2022MSAV}.

In this paper we consider the matched-density CHNS equations on the rectangle
\(\Omega=(0,L_x)\times(0,L_y)\):
\begin{subequations}\label{eq:continuous-chns}
	\begin{align}
		\partial_t\phi + \nabla\cdot(\phi\bfu) - \Delta\mu                               & =0, \label{eq:continuous-phase}                                          \\
		\mu                                                                              & = -\eps^2\Delta\phi + f(\phi),\quad f(s)=s^3-s, \label{eq:continuous-mu} \\
		\partial_t\bfu + (\bfu\cdot\nabla)\bfu -\nu\Delta\bfu + \nabla p + \phi\nabla\mu & =0, \label{eq:continuous-momentum}                                       \\
		\nabla\cdot\bfu                                                                  & =0. \label{eq:continuous-div}
	\end{align}
\end{subequations}
We impose the homogeneous boundary conditions
\begin{equation}\label{eq:continuous-bc}
	\bfu=0,
	\quad \partial_{\mathbf n}\phi=\partial_{\mathbf n}\mu=0
	\quad\text{on }\partial\Omega,
\end{equation}
and normalize the pressure by \(\int_\Omega p\,dx=0\). This normalization fixes
the additive pressure constant and is consistent with the compatibility condition
for the Neumann pressure Poisson equation.
Its free energy is
\begin{equation}\label{eq:continuous-energy}
	\calE(\phi,\bfu)=\int_\Omega
	\left\{\tfrac12|\bfu|^2+\tfrac{\eps^2}{2}|\nabla\phi|^2+F(\phi)\right\}\,dx,
	\quad
	F(s)=\tfrac14(s^2-1)^2.
\end{equation}
Testing the phase equation by \(\mu\), the chemical-potential relation by
\(\partial_t\phi\), and the momentum equation by \(\bfu\) gives
\begin{equation}\label{eq:continuous-energy-law}
	\tfrac{d}{dt}\calE(\phi,\bfu)+\norm{\nabla\mu}_{L^2}^2
	+\nu\norm{\nabla\bfu}_{L^2}^2=0.
\end{equation}
Here \((\nabla\cdot(\phi\bfu),\mu)=-(\phi\bfu,\nabla\mu)\), which cancels
\((\phi\nabla\mu,\bfu)\) in the momentum equation. The boundary conditions
\eqref{eq:continuous-bc} remove the boundary contributions in the viscous,
chemical-potential, and transport terms.

This work focuses on BDF2 convex-splitting discretizations. For the
Cahn--Hilliard equation, existing nonlinear BDF2 treatments use two related
forms of the convex-splitting force. The first is the fully implicit force,
which treats both the cubic and the linear concave part at the new time:
\begin{equation}\label{eq:intro-existing-bdf2-forces}
	(\phi^{n+1})^3-\phi^{n+1}.
\end{equation}
This force \eqref{eq:intro-existing-bdf2-forces} appears in the variable-step
BDF2 discretization for nonlocal Cahn--Hilliard models and space-fractional
variants \cite{XueZhaiZhao2024}. The second is the extrapolated force:
\begin{equation}\label{eq:intro-existing-bdf2-extrapolated-force}
	(\phi^{n+1})^3-\phi^{*,n+1},
	\quad \phi^{*,n+1}=2\phi^n-\phi^{n-1}.
\end{equation}
This force \eqref{eq:intro-existing-bdf2-extrapolated-force}, in the cited
schemes, closes the energy estimate with an additional
stabilization parameter
\cite{YanChenWangWise2018,ChenWangYanZhang2019,LiaoJiWangZhang2022,HuCheng2024}.
The choices \eqref{eq:intro-existing-bdf2-forces} and
\eqref{eq:intro-existing-bdf2-extrapolated-force} are effective for the scalar
Cahn--Hilliard gradient flow, but the available energy estimates rely on the
closed \(H^{-1}\) structure of the phase equation.
Schematically, the corresponding scalar BDF2 update has the form
\begin{equation}\label{eq:intro-pure-ch-time}
	D_{2\tau}\phi^{n+1}=\Delta\mu^{n+1},
	\quad
	\mu^{n+1}=-\eps^2\Delta\phi^{n+1}+g^{n+1},
\end{equation}
where
\[
	D_{2\tau}\phi^{n+1}=\tfrac{3\phi^{n+1}-4\phi^n+\phi^{n-1}}{2\dt}.
\]
Here \(g^{n+1}\) denotes one of the nonlinear approximations in
\eqref{eq:intro-existing-bdf2-forces} or
\eqref{eq:intro-existing-bdf2-extrapolated-force}, with stabilization when
needed. For the scalar Cahn--Hilliard update \eqref{eq:intro-pure-ch-time}, the
energy proof can proceed in the \(H^{-1}\) metric: introduce \(\psi^{n+1}\) by
\(-\Delta\psi^{n+1}=D_{2\tau}\phi^{n+1}\), test the phase equation by
\(\psi^{n+1}\), and test the chemical-potential relation by
\(D_{2\tau}\phi^{n+1}\). Adding the two identities gives the \(H^{-1}\)
dissipation together with the BDF2 bulk-energy difference. In CHNS this scalar
argument is not closed, because the phase equation is replaced by the
transported equation
\begin{equation}\label{eq:intro-chns-phase-time}
	D_{2\tau}\phi^{n+1}
	+\nabla\cdot(\phi^{*,n+1}\bfu^{n+1})
	=\Delta\mu^{n+1} .
\end{equation}
The CHNS energy estimate must instead be a coupled estimate: the transported
phase equation \eqref{eq:intro-chns-phase-time} is tested by \(\mu^{n+1}\), the
momentum equation is tested by
\(\bfu^{n+1}\), and the transport contribution is paired with the capillary
force through
\begin{equation}\label{eq:intro-transport-capillary-cancel-time}
	\bigl(\nabla\cdot(\phi^{*,n+1}\bfu^{n+1}),\mu^{n+1}\bigr)
	+\bigl(\phi^{*,n+1}\nabla\mu^{n+1},\bfu^{n+1}\bigr)=0 .
\end{equation}
For this reason, BDF2 convex-splitting estimates for the Cahn--Hilliard
equation do not transfer directly to CHNS. The nonlinear part of the chemical
potential must produce a BDF2 bulk-energy difference under the pairing with
\(\dt D_{2\tau}\phi^{n+1}\), while the phase and momentum equations must retain
the transport--capillary cancellation
\eqref{eq:intro-transport-capillary-cancel-time}.

Motivated by this energy-pairing requirement, we modify the discretization of
the bulk force in the chemical potential. This is the point at which the
present scheme differs from the BDF2 Cahn--Hilliard forces in
\eqref{eq:intro-existing-bdf2-forces} and
\eqref{eq:intro-existing-bdf2-extrapolated-force}: the cubic part is evaluated
through two shifted BDF2 states chosen so that their difference is exactly the
BDF2 phase increment. Define
\begin{equation}\label{eq:intro-shifted-bdf2-difference-time}
	\phi^{\dagger,n+1}:=\tfrac{3\phi^{n+1}-\phi^n}{2},
	\quad
	\phi^{\dagger,n}:=\tfrac{3\phi^n-\phi^{n-1}}{2},
	\quad
	\dt D_{2\tau}\phi^{n+1}=\phi^{\dagger,n+1}-\phi^{\dagger,n} .
\end{equation}
Consequently, in the bulk force \eqref{eq:intro-bdf2-force-time}, the two
arguments in the quartic part are chosen as the shifted BDF2 states
\(\phi^{\dagger,n+1}\) and \(\phi^{\dagger,n}\), not \(\phi^{n+1}\) and
\(\phi^n\). With
\[
	F(s)=\tfrac14(s^2-1)^2,
	\quad
	\chi(a,b)=\tfrac14(a^2+b^2)(a+b),
\]
the bulk force used here is
\begin{equation}\label{eq:intro-bdf2-force-time}
	\chi(\phi^{\dagger,n+1},\phi^{\dagger,n})-\phi^{*,n+1} .
\end{equation}
This is a second-order approximation of \(\phi^3-\phi\). It satisfies the
following exact BDF2 bulk-energy identity:
\begin{align}\label{eq:intro-bdf2-local-chain-time}
	 & \left(
	\chi(\phi^{\dagger,n+1},\phi^{\dagger,n})-\phi^{*,n+1},
	\dt D_{2\tau}\phi^{n+1}
	\right) \notag                                                  \\
	 & \quad=
	\left(F(\phi^{\dagger,n+1})-F(\phi^{\dagger,n}),1\right) \notag \\
	 & \quad
	+\tfrac38\left(
	\norm{\phi^{n+1}-\phi^n}^2
	-\norm{\phi^n-\phi^{n-1}}^2
	\right)
	+\tfrac34\norm{\phi^{n+1}-2\phi^n+\phi^{n-1}}^2 .
\end{align}
The identity \eqref{eq:intro-bdf2-local-chain-time} is the key ingredient
that is not obtained by directly inserting the scalar BDF2 Cahn--Hilliard
forces
\eqref{eq:intro-existing-bdf2-forces}--\eqref{eq:intro-existing-bdf2-extrapolated-force}
into the coupled CHNS estimate. In the scalar Cahn--Hilliard equation, the
BDF2 energy argument can be closed in the \(H^{-1}\) metric: the phase
equation is tested by an inverse-Laplacian increment, while the
chemical-potential equation is paired with the one-step increment
\(\phi^{n+1}-\phi^n\). This testing is compatible with the gradient-flow
structure of scalar CH, but it does not preserve the transport--capillary
cancellation in CHNS. For CHNS, the phase equation must instead be tested by \(\mu^{n+1}\), so that the capillary work cancels with the momentum equation as in
\eqref{eq:intro-transport-capillary-cancel-time}. Consequently, the
chemical-potential relation itself must provide the local BDF2 chain rule
\eqref{eq:intro-bdf2-local-chain-time} when paired with
\(\dt D_{2\tau}\phi^{n+1}\). This requirement motivates the force
\eqref{eq:intro-bdf2-force-time}. It is a shifted-state BDF2 analogue of a
two-point discrete-gradient treatment of the double-well potential, based on
\eqref{eq:intro-shifted-bdf2-difference-time}. The \(\chi\)-term is the
discrete gradient of the convex quartic part between the shifted states, while
the extrapolated term \(-\phi^{*,n+1}\) treats the concave quadratic part.
Together they yield \eqref{eq:intro-bdf2-local-chain-time}, with a telescoping
BDF2 history term and a nonnegative remainder.

We then discretize time by CS-BDF2 and use a staggered-grid spatial discretization
on a bounded rectangular grid. The scheme conserves mass, is uniquely solvable after fixing the pressure
mean, and satisfies an unconditional discrete energy law. We prove second-order
convergence for the two-dimensional setting considered here.

The rest of the paper is organized as follows. Section~2 introduces the notations of  staggered grid
and some useful lemmas. Section~3 presents the scheme and its solvability, energy stability. Section~4 gives the convergence result. Section~5 reports numerical
experiments, and Section~6 concludes.

\section{Preliminaries}\label{sec:preliminaries}

Let \(\Omega=(0,L_x)\times(0,L_y)\), \(h_x=L_x/N_x\), and
\(h_y=L_y/N_y\). Following
\cite{HarlowWelch1965,WeiserWheeler1988,GongZhaoWang2018,RuiLi2017,LiShen2020NS,LiShen2020},
we define the following sets of grid functions on the staggered-grid
\begin{equation*}
	\begin{aligned}
		\mathsf C_h
		=\{\varphi_{i+\frac12,j+\frac12}:0\le i<N_x,
		0\le j<N_y\},                                              \
		\mathsf V_h
		=\{g_{i,j}:0\le i\le N_x,
		0\le j\le N_y\}, \\
		\mathsf E_h^x
		=\{u_{i,j+\tfrac12}:1\le i<N_x,
		0\le j<N_y\},                                              \
		\mathsf E_h^y
		=\{v_{i+\tfrac12,j}:0\le i<N_x,
		1\le j<N_y\}.    \\
	\end{aligned}
\end{equation*}
We further define the product space $\mathsf E_h = \mathsf E_h^x \times \mathsf E_h^y$  The phase variable $\phi_h$, chemical potential $\mu_h$, and pressure $p_h$ are defined on $\mathsf C_h$, while the velocity
$\mathbf u_h=(u_{1,h},u_{2,h}) \in \mathsf E_h$.
Throughout the paper, velocity boundary and ghost values are chosen to enforce
the no-slip condition, whereas scalar boundary and ghost values are determined
by the no-flux condition.

For \(u_h\in\mathsf E_h^x\), \(v_h\in\mathsf E_h^y\),
\(g_h\in\mathsf V_h\), and \(\varphi_h\in\mathsf C_h\), define the following difference quotient operator
\begin{equation}\label{eq:mac-averages}
	\begin{aligned}
		 & (a_xu_h)_{i+\frac12,j+\frac12}
		=\tfrac{u_{i+1,j+\frac12}+u_{i,j+\frac12}}{2},
		 (d_xu_h)_{i+\frac12,j+\frac12}
		=\tfrac{u_{i+1,j+\frac12}-u_{i,j+\frac12}}{h_x},
		 (A_x\varphi_h)_{i,j+\frac12}
		=\tfrac{\varphi_{i+\frac12,j+\frac12}
			 +\varphi_{i-\frac12,j+\frac12}}{2},          \\
		 & (D_x\varphi_h)_{i,j+\frac12}
		=\tfrac{\varphi_{i+\frac12,j+\frac12}
			 -\varphi_{i-\frac12,j+\frac12}}{h_x},
		 (A_xv_h)_{i,j}
		=\tfrac{v_{i+\frac12,j}+v_{i-\frac12,j}}{2},
		 (D_xv_h)_{i,j}
		=\tfrac{v_{i+\frac12,j}-v_{i-\frac12,j}}{h_x}, \\[2mm]
		 & (a_yv_h)_{i+\frac12,j+\frac12}
		=\tfrac{v_{i+\frac12,j+1}+v_{i+\frac12,j}}{2},
		 (d_yv_h)_{i+\frac12,j+\frac12}
		=\tfrac{v_{i+\frac12,j+1}-v_{i+\frac12,j}}{h_y},
		 (A_y\varphi_h)_{i+\frac12,j}
		=\tfrac{\varphi_{i+\frac12,j+\frac12}
			 +\varphi_{i+\frac12,j-\frac12}}{2},          \\
		 & (D_y\varphi_h)_{i+\frac12,j}
		=\tfrac{\varphi_{i+\frac12,j+\frac12}
			 -\varphi_{i+\frac12,j-\frac12}}{h_y},
		 (A_yu_h)_{i,j}
		=\tfrac{u_{i,j+\frac12}+u_{i,j-\frac12}}{2},
		 (D_yu_h)_{i,j}
		=\tfrac{u_{i,j+\frac12}-u_{i,j-\frac12}}{h_y}, \\[2mm]
		 & (a_xg_h)_{i+\frac12,j}
		=\tfrac{g_{i+1,j}+g_{i,j}}{2},
		 (d_xg_h)_{i+\frac12,j}
		=\tfrac{g_{i+1,j}-g_{i,j}}{h_x},
		 (a_yg_h)_{i,j+\frac12}
		=\tfrac{g_{i,j+1}+g_{i,j}}{2},                 \\
		 & (d_yg_h)_{i,j+\frac12}
		=\tfrac{g_{i,j+1}-g_{i,j}}{h_y}.
	\end{aligned}
\end{equation}
Set
\begin{equation}\label{eq:mac-grad-div}
	\begin{aligned}
		D_h\varphi_h & =(D_x\varphi_h,D_y\varphi_h),\quad
		d_h\bfv_h=d_xv_{1,h}+d_yv_{2,h},\quad
		A_h\varphi_h=-d_hD_h\varphi_h,                    \\
		A_h\bfv_h    & =\bigl(-(D_xd_x+d_yD_y)v_{1,h},
		-(d_xD_x+D_yd_y)v_{2,h}\bigr) .
	\end{aligned}
\end{equation}
For scalar-vector products,
\begin{equation}\label{eq:mac-product-def}
	(\varphi_h\bfv_h)_{1,i,j+\frac12}
	=(A_x\varphi_h)_{i,j+\tfrac12}v_{1,i,j+\frac12}, \quad
	(\varphi_h\bfv_h)_{2,i+\frac12,j}
	=(A_y\varphi_h)_{i+\frac12,j}v_{2,i+\frac12,j}.
\end{equation}

The discrete inner products are
\begin{equation}\label{eq:mac-inner-products}
	\begin{aligned}
		(\varphi_h,\psi_h)_c
		 & =h_xh_y\sum_{i,j}\varphi_{i+\frac12,j+\frac12}
		\psi_{i+\frac12,j+\frac12},
		 &
		(r_h,s_h)_x
		 & =h_xh_y\sum_{i,j}r_{i,j+\frac12}s_{i,j+\frac12},\notag \\
		(r_h,s_h)_y
		 & =h_xh_y\sum_{i,j}r_{i+\frac12,j}s_{i+\frac12,j},
		 &
		(r_h,s_h)_v
		 & =h_xh_y\sum_{i,j}r_{i,j}s_{i,j}.
	\end{aligned}
\end{equation}
For vectors, \((\bfu_h,\bfv_h)_h=(u_{1,h},v_{1,h})_x+(u_{2,h},v_{2,h})_y\);
for scalars, \((\cdot,\cdot)_h=(\cdot,\cdot)_c\).
Set \(\mathsf C_{h,0}:=\{\varphi_h\in\mathsf C_h:(\varphi_h,1)_h=0\}\).
For a scalar grid function \(q_h\), with the index range understood from its grid location, set
\begin{equation}\label{eq:mac-lp-norms}
	\norm{q_h}_{\ell^p}^p=h_xh_y\sum_{i,j}|q_{i,j}|^p,
	\ 1\le p<\infty,
	\quad
	\norm{q_h}_{\ell^\infty}=\max_{i,j}|q_{i,j}|.
\end{equation}
For \(\boldsymbol z_h=(z_{1,h},z_{2,h})\in\mathsf E_h\), define
\begin{equation}\label{eq:mac-vector-lp-norms}
	\norm{\boldsymbol z_h}_{\ell^p}^p
	=\norm{z_{1,h}}_{\ell^p}^p+\norm{z_{2,h}}_{\ell^p}^p,
	\ 1\le p<\infty,
	\quad
	\norm{\boldsymbol z_h}_{\ell^\infty}
	=\max\{\norm{z_{1,h}}_{\ell^\infty},\norm{z_{2,h}}_{\ell^\infty}\}.
\end{equation}
We further introduce
\begin{equation}\label{eq:mac-vector-h1}
	\norm{D_h\bfv_h}_h^2
	=\norm{d_xv_{1,h}}_c^2+\norm{D_yv_{1,h}}_v^2
	+\norm{D_xv_{2,h}}_v^2+\norm{d_yv_{2,h}}_c^2 .
\end{equation}
For \(\varphi_h\in\mathsf C_{h,0}\) and
\(\bfv_h\in\mathsf E_h\), define the negative norms by
\begin{equation}\label{eq:hminusone-def}
	\norm{\varphi_h}_{H_h^{-1}}
	=\sup_{\psi_h\in\mathsf C_{h,0},\,\psi_h\ne0}
	\tfrac{(\varphi_h,\psi_h)_h}{\norm{D_h\psi_h}_h},\quad
	\norm{\bfv_h}_{H_h^{-1}}
	=\sup_{\boldsymbol w_h\ne0}
	\tfrac{(\bfv_h,\boldsymbol w_h)_h}{\norm{D_h\boldsymbol w_h}_h}.
\end{equation}

\begin{lemma}[MAC inf-sup condition, \cite{RuiLi2017,LiShen2020NS}]\label{lem:mac-infsup}
	There is a constant \(\beta>0\), independent of \(h\), such that
	\begin{equation*}
		\sup_{\bfv_h\ne0}\tfrac{(q_h,d_h\bfv_h)_h}{\norm{D_h\bfv_h}_h}
		\ge \beta\norm{q_h}_h
		\quad\forall q_h\in\mathsf C_{h,0}.
	\end{equation*}
\end{lemma}

\begin{lemma}[\cite{HarlowWelch1965,WeiserWheeler1988,GongZhaoWang2018,LiShen2020}]\label{lem:mac-sbp-identities}
	With assumed boundary condition, the following summation-by-parts
	identities hold:
	\begin{equation*}
		(D_h\varphi_h,\bfv_h)_h =-(\varphi_h,d_h\bfv_h)_h, \
		(A_h\varphi_h,\psi_h)_h           =(D_h\varphi_h,D_h\psi_h)_h, \
		(d_h(\varphi_h\bfv_h),\psi_h)_h  =-(\varphi_h\bfv_h,D_h\psi_h)_h.
	\end{equation*}
\end{lemma}

\begin{lemma}[\cite{HarlowWelch1965,GongZhaoWang2018,LiShen2020NS,LiShen2020}]
	\label{lem:mac-convection-form}
	For \(\bfu_h, \bfv_h \in \mathsf E_h \), define
	$\mathcal B_h(\bfu_h,\bfv_h)$ by
	\begin{equation*}
		\mathcal B_h (\bfu_h, \bfv_h) =
		\tfrac12
		\begin{pmatrix}
			u_{1,h}D_x(a_x v_{1,h})
			+A_x\bigl(d_x(u_{1,h}v_{1,h})\bigr) +a_y\bigl(A_x u_{2,h}D_y v_{1,h}\bigr) +d_y\bigl(A_y v_{1,h}A_x u_{2,h}\bigr) \\
			a_x\bigl(A_y u_{1,h}D_x v_{2,h}\bigr)
			+d_x\bigl(A_y u_{1,h}A_x v_{2,h}\bigr) + u_{2,h}D_y(a_y v_{2,h})
			+A_y\bigl(d_y(u_{2,h}v_{2,h})\bigr).
		\end{pmatrix}
	\end{equation*}
	Let
	\[
		b_h(\bfu_h,\bfv_h,\bfw_h)
		:=(\mathcal B_h(\bfu_h,\bfv_h),\bfw_h)_h .
	\]
	It is readily verified that if \(d_h\bfu_h=0\), then
	\begin{equation*}
		b_h(\bfu_h,\bfv_h,\bfw_h)=-b_h(\bfu_h,\bfw_h,\bfv_h),
		\quad
		b_h(\bfu_h,\bfv_h,\bfv_h)=0
		\ \text{for all }\bfv_h,\bfw_h\in\mathsf E_h .
	\end{equation*}
\end{lemma}

\begin{lemma}\label{lem:mac-convection-estimates}
	In 2D, if \(d_h\mathbf u_h=0\), then
	\begin{equation*}
		\begin{aligned}
			|b_h(\bfu_h,\bfv_h,\bfw_h)|
			 & \le C\norm{\bfu_h}_h^{1/2}\norm{D_h\bfu_h}_h^{1/2}
			\norm{\bfv_h}_h^{1/2}\norm{D_h\bfv_h}_h^{1/2}
			\norm{D_h\bfw_h}_h,
			\\
			\norm{\mathcal B_h(\bfu_h,\bfv_h)}_{H_h^{-1}}
			 & \le C\norm{\bfu_h}_{\ell^\infty}\norm{\bfv_h}_h,
			\\
			\norm{\mathcal B_h(\bfu_h,\bfv_h)}_{H_h^{-1}}
			 & \le C\left(\norm{\bfv_h}_{\ell^\infty}
			+\norm{D_h\bfv_h}_{\ell^\infty}\right)\norm{\bfu_h}_h,
			\\
			\norm{\mathcal B_h(\bfu_h,\bfv_h)}_{H_h^{-1}}
			 & \le C\norm{\bfu_h}_{\ell^4}\norm{\bfv_h}_{\ell^4}.
		\end{aligned}
	\end{equation*}
	These estimates are discrete counterparts of the continuous bounds in
	\cite{HuangShen2024NSCH}. The proof is given in Appendix~\ref{app:mac-convection-estimates}.
\end{lemma}

\begin{lemma}[\cite{WeiserWheeler1988,GuoWangWiseYue2016,RuiLi2017,LiShenRui2019,LiShen2020NS,LiShen2020}]\label{lem:mac-discrete-estimates}
	The following estimates hold uniformly in (h). For
	$\overline{\varphi_h}=0$ and $\mathbf u_h$ satisfying the homogeneous
	velocity boundary condition, and for any fixed $2\le p<\infty$,
	\[
		\begin{aligned}
			\norm{\varphi_h}_h & \le C\norm{D_h\varphi_h}_h,\
			\norm{\bfv_h}_h \le C\norm{D_h\bfv_h}_h,                                              \\
			\norm{\varphi_h}_{\ell^p}
			                   & \le C_p\left(\norm{D_h\varphi_h}_h+\norm{\varphi_h}_h\right),    \\
			\norm{\varphi_h}_{\ell^\infty}
			                   & \le C\left(\norm{A_h\varphi_h}_h+\norm{\varphi_h}_h\right),      \\
			\norm{\varphi_h}_{\ell^p}
			                   & \le C_p\norm{\varphi_h}_h^{2/p}\norm{D_h\varphi_h}_h^{1-2/p},    \\
			\norm{\bfv_h}_{\ell^p}
			                   & \le C_p\norm{\bfv_h}_h^{2/p}\norm{D_h\bfv_h}_h^{1-2/p},          \\
			\norm{D_h\varphi_h}_{\ell^p}
			                   & \le C_p\norm{D_h\varphi_h}_h^{2/p}\norm{A_h\varphi_h}_h^{1-2/p}.
		\end{aligned}
	\]
\end{lemma}

Let $T>0$, $N\ge2$, $\tau=T/N$, and $t_n=n\tau$ for
$0\le n\le N$. For a sequence time grid function $q^n$, define
\begin{equation}\label{eq:bdf2-defs}
	\begin{aligned}
		D_{2\tau} q^{n+1}
		=\tfrac{3q^{n+1}-4q^n+q^{n-1}}{2\dt},
		\delta^2 q^{n+1}
		=q^{n+1}-2q^n+q^{n-1}. \\
		q^{*,n+1}
		=2q^n-q^{n-1},
		q^{\dagger,n+1}
		=\tfrac{3q^{n+1}-q^n}{2},
		q^{\dagger,n}
		=\tfrac{3q^n-q^{n-1}}{2},
	\end{aligned}
\end{equation}
For $(n=1)$, we define $D_{2\tau}q^1=\frac{q^1-q^0}{\tau}$. The start-up approximation is not specified explicitly and is only required to
satisfy Assumption~\ref{ass:start-up}.

Define
\begin{equation}\label{eq:g-functional}
	\calG(a,b)
	:=\frac14\left(\norm{a}_h^2+\norm{2a-b}_h^2\right),
\end{equation}
We also introduce the cubic function
\[
	\chi(a,b)
	:=\frac14(a^2+b^2)(a+b).
\]

\begin{lemma}\label{lem:bdf2-algebra}
	For any grid function \(q^n\),
	\begin{equation*}
		\dt\,(D_{2\tau}q^{n+1},q^{n+1})_h
		=\calG(q^{n+1},q^n)-\calG(q^n,q^{n-1})
		+\tfrac14\norm{\delta^2q^{n+1}}_h^2 .
	\end{equation*}
	If \(q^n\) is scalar-valued, then
	\begin{align*}
		 & \left(\chi(q^{\dagger,n+1},q^{\dagger,n})-q^{*,n+1},
		\dt D_{2\tau}q^{n+1}\right)_h                                                            \\
		 & \quad=(F(q^{\dagger,n+1})-F(q^{\dagger,n}),1)_h +\tfrac38\left(\norm{q^{n+1}-q^n}_h^2
		-\norm{q^n-q^{n-1}}_h^2\right)
		+\tfrac34\norm{\delta^2q^{n+1}}_h^2 .
	\end{align*}
\end{lemma}

\begin{proof}
	The first identity in Lemma~\ref{lem:bdf2-algebra} is the BDF2
	\(G\)-stability identity \cite{Dahlquist1978}. For the second identity,
	\eqref{eq:bdf2-defs} implies
	\[
		q^{\dagger,n+1}-q^{\dagger,n}=\dt D_{2\tau}q^{n+1},
		\quad
		\tfrac{q^{\dagger,n+1}+q^{\dagger,n}}2-q^{*,n+1}
		=\tfrac34\delta^2q^{n+1}.
	\]
	Hence,
	\begin{align*}
		 & \left(\chi(q^{\dagger,n+1},q^{\dagger,n})-q^{*,n+1}\right)
		\dt D_{2\tau}q^{n+1}                                          \\
		 & \quad=F(q^{\dagger,n+1})-F(q^{\dagger,n})
		+\left(\tfrac{q^{\dagger,n+1}+q^{\dagger,n}}2-q^{*,n+1}\right)
		\dt D_{2\tau}q^{n+1}                                          \\
		 & \quad=F(q^{\dagger,n+1})-F(q^{\dagger,n})
		+\tfrac38\,\delta^2q^{n+1}
		(3q^{n+1}-4q^n+q^{n-1})                                       \\
		 & \quad=F(q^{\dagger,n+1})-F(q^{\dagger,n})
		+\tfrac38\big((q^{n+1}-q^n)^2-(q^n-q^{n-1})^2\big)
		+\tfrac34(\delta^2q^{n+1})^2 .
	\end{align*}
	Summing over the grid yields the claim.
\end{proof}

\begin{lemma}\label{lem:chi-estimates}
	For \(a,b,c\in\mathbb R\),
	\[
		\chi(a,b)=\chi(b,a),\qquad
		(\chi(a,c)-\chi(b,c))(a-b)\ge0,
	\]
	and
	\[
		\chi(a,b)-\chi(b,c)
		=\frac14(a-c)\bigl(a^2+ac+c^2+b(a+c)+b^2\bigr).
	\]
	For cell-centered grid functions \(u_h,v_h,w_h,z_h,s_h\),
	\[
		\begin{aligned}
			\left|(\chi(u_h,v_h)-\chi(v_h,w_h),z_h)_h\right|
			 & \le
			\frac14\norm{u_h-w_h}_h\norm{z_h}_{\ell^6}
			\left(
			\norm{u_h}_{\ell^6}
			+\norm{v_h}_{\ell^6}
			+\norm{w_h}_{\ell^6}
			\right)^2, \\
			\norm{\chi(u_h,v_h)-s_h}_h
			 & \le
			\frac14
			\left(
			\norm{u_h}_{\ell^6}
			+\norm{v_h}_{\ell^6}
			\right)^3
			+\norm{s_h}_h .
		\end{aligned}
	\]
	Let \(r_h\) be cell-centered and set
	\[
		\begin{aligned}
			K_\infty
			 & =\norm{u_h}_{\ell^\infty}
			+\norm{v_h}_{\ell^\infty}
			+\norm{w_h}_{\ell^\infty}
			+\norm{r_h}_{\ell^\infty},   \\
			K_4
			 & =\norm{D_hu_h}_{\ell^4}
			+\norm{D_hv_h}_{\ell^4}
			+\norm{D_hw_h}_{\ell^4}
			+\norm{D_hr_h}_{\ell^4}.
		\end{aligned}
	\]
	If \(u_h-w_h\) and \(v_h-r_h\) have zero mean, then
	\[
		\norm{D_h\bigl(\chi(u_h,v_h)-\chi(w_h,r_h)\bigr)}_h
		\le
		\left(
		\frac14K_\infty^2+\frac12C_4K_\infty K_4
		\right)
		\left(
		\norm{D_h(u_h-w_h)}_h
		+\norm{D_h(v_h-r_h)}_h
		\right),
	\]
	where \(C_4\) is the constant in the \(p=4\) estimate of
	Lemma~\ref{lem:mac-discrete-estimates}.
\end{lemma}

\begin{proof}
	The symmetry is immediate. For fixed \(c\),
	\[
		\partial_1\chi(a,c)
		=\frac14(3a^2+2ac+c^2)
		=\frac14\bigl(2a^2+(a+c)^2\bigr)\ge0,
	\]
	which gives the monotonicity. The factorization follows by expanding the two
	cubic polynomials.

	Set
	\[
		\mathcal P_h
		=u_h^2+u_hw_h+w_h^2+v_h(u_h+w_h)+v_h^2.
	\]
	Then
	\[
		\chi(u_h,v_h)-\chi(v_h,w_h)
		=\frac14(u_h-w_h)\mathcal P_h,
	\]
	and
	\[
		|\mathcal P_h|
		\le
		\bigl(|u_h|+|v_h|+|w_h|\bigr)^2.
	\]
	H\"older's inequality gives
	\[
		\begin{aligned}
			\left|(\chi(u_h,v_h)-\chi(v_h,w_h),z_h)_h\right|
			 & \le
			\frac14\norm{u_h-w_h}_h\norm{\mathcal P_h}_{\ell^3}
			\norm{z_h}_{\ell^6} \\
			 & \le
			\frac14\norm{u_h-w_h}_h\norm{z_h}_{\ell^6}
			\left(
			\norm{u_h}_{\ell^6}
			+\norm{v_h}_{\ell^6}
			+\norm{w_h}_{\ell^6}
			\right)^2 .
		\end{aligned}
	\]
	Also,
	\[
		|\chi(u_h,v_h)|
		\le
		\frac14\bigl(|u_h|+|v_h|\bigr)^3,
	\]
	and hence
	\[
		\norm{\chi(u_h,v_h)-s_h}_h
		\le
		\frac14
		\left(
		\norm{u_h}_{\ell^6}
		+\norm{v_h}_{\ell^6}
		\right)^3
		+\norm{s_h}_h .
	\]

	For the gradient estimate, write
	\[
		\begin{aligned}
			\chi(u_h,v_h)-\chi(w_h,r_h)
			 & =\frac14(u_h-w_h)\mathcal P_h+\frac14(v_h-r_h)\mathcal Q_h,
		\end{aligned}
	\]
	where
	\[
		\begin{aligned}
			\mathcal P_h
			 & =u_h^2+u_hw_h+w_h^2+v_h(u_h+w_h)+v_h^2,  \\
			\mathcal Q_h
			 & =w_h^2+w_h(v_h+r_h)+v_h^2+v_hr_h+r_h^2 .
		\end{aligned}
	\]
	By the definitions of \(K_\infty\) and \(K_4\),
	\[
		\norm{\mathcal P_h}_{\ell^\infty},
		\norm{\mathcal Q_h}_{\ell^\infty}
		\le K_\infty^2,
	\]
	and
	\[
		\norm{D_h\mathcal P_h}_{\ell^4},
		\norm{D_h\mathcal Q_h}_{\ell^4}
		\le 2K_\infty K_4.
	\]
	The discrete product rule and H\"older's inequality yield
	\[
		\begin{aligned}
			\norm{D_h\bigl(\chi(u_h,v_h)-\chi(w_h,r_h)\bigr)}_h
			 & \le
			\frac14K_\infty^2
			\left(
			\norm{D_h(u_h-w_h)}_h
			+\norm{D_h(v_h-r_h)}_h
			\right)  \\
			 & \quad
			+\frac12K_\infty K_4
			\left(
			\norm{u_h-w_h}_{\ell^4}
			+\norm{v_h-r_h}_{\ell^4}
			\right).
		\end{aligned}
	\]
	Since \(u_h-w_h\) and \(v_h-r_h\) have zero mean,
	Lemma~\ref{lem:mac-discrete-estimates} with \(p=4\) gives
	\[
		\norm{u_h-w_h}_{\ell^4}
		+\norm{v_h-r_h}_{\ell^4}
		\le
		C_4
		\left(
		\norm{D_h(u_h-w_h)}_h
		+\norm{D_h(v_h-r_h)}_h
		\right).
	\]
	The desired estimate follows.
\end{proof}

\section{The CS-BDF2 scheme}

Given the two previous time levels
\[
	(\phi_h^j,\bfu_h^j,p_h^j),\quad j=n,n-1,
\]
find \((\phi_h^{n+1},\mu_h^{n+1},\bfu_h^{n+1},p_h^{n+1})\) such that
\begin{subequations}\label{eq:scheme}
	\begin{align}
		D_{2\tau}\bfu_h^{n+1}
		+\mathcal B_h(\bfu_h^{*,n+1},\bfu_h^{n+1})
		+\nu A_h\bfu_h^{n+1}+D_hp_h^{n+1}
		+\phi_h^{*,n+1}D_h\mu_h^{n+1}    & =0, \label{eq:scheme-u}   \\
		d_h\bfu_h^{n+1}                  & =0, \label{eq:scheme-div} \\
		D_{2\tau}\phi_h^{n+1}+A_h\mu_h^{n+1}
		+d_h(\phi_h^{*,n+1}\bfu_h^{n+1}) & =0, \label{eq:scheme-phi} \\
		\mu_h^{n+1}-\eps^2A_h\phi_h^{n+1}
		-\chi(\phi_h^{\dagger,n+1},\phi_h^{\dagger,n})
		+\phi_h^{*,n+1}                  & =0. \label{eq:scheme-mu}
	\end{align}
\end{subequations}
The initial data at \(t=0\) are obtained by the mass-corrected restriction,
the MAC velocity projection, and the restricted exact chemical potential.
The values at \(t=\tau\) are supplied by a second-order, mass-compatible
start-up procedure satisfying Assumption~\ref{ass:start-up}.

\begin{lemma}[Brouwer fixed point theorem \cite{Brower1, Brower2}]
	\label{lem:brouwer}
	Let \(V\) be a finite-dimensional real Hilbert space, and let
	\(\mathcal R:V\to V'\) be continuous. Identify \(V'\) with \(V\)
	through the Riesz map. Suppose that, for some \(R>0\),
	\[
		\langle \mathcal R(z),z\rangle \ge 0
		\quad
		\forall z\in V,\quad \norm{z}_V=R .
	\]
	Then there exists \(z_R\in V\), \(\norm{z_R}_V\le R\), such that $\mathcal R(z_R)=0$.
\end{lemma}

We first prove mass conservation and solvability of the fully discrete scheme.

\begin{theorem}\label{thm:mass-solvability}
	Assume $(\phi_h^1,1)_h=(\phi_h^0,1)_h$ and $d_h\bfu_h^0=d_h\bfu_h^1=0$.
	Then, for each \(1\le n\le N-1\), the CS-BDF2 scheme \eqref{eq:scheme}, with $(p_h^{n+1},1)_h=0$
	admits a unique solution
	\((\bfu_h^{n+1},p_h^{n+1},\phi_h^{n+1},\mu_h^{n+1})\).
	Moreover, it preserves the mass:
	\begin{equation}\label{eq:mass-conservation}
		(\phi_h^{n+1},1)_h
		=
		(\phi_h^n,1)_h
		=
		(\phi_h^0,1)_h .
	\end{equation}
\end{theorem}

\begin{proof}
	Testing \eqref{eq:scheme-phi} by \(1\) and utilizing $(A_h\mu_h^{n+1},1)_h=0,
		\ (d_h(\phi_h^{*,n+1}\bfu_h^{n+1}),1)_h=0$ gives
	\[
		3(\phi_h^{n+1},1)_h
		-4(\phi_h^n,1)_h
		+(\phi_h^{n-1},1)_h=0,
	\]
	since $(\phi_h^1,1)_h=(\phi_h^0,1)_h$, we obtain \eqref{eq:mass-conservation} by induction.

	We now prove existence at a fixed time step \(n\), assuming that the two
	previous levels have already been constructed. Set $m_0=(\phi_h^0,1)_h$
	and introduce
	\[
		\mathsf U_h
		=
		\{\bfv_h\in\mathsf E_h:\ d_h\bfv_h=0\}.
	\]
	We work in the finite-dimensional Hilbert space $\mathbb V_h = \mathsf U_h\times\mathsf C_{h,0}\times\mathsf C_{h,0}.$
	We use the norm
	\[
		\norm{(\bfv_h,\theta_h,\zeta_h)}_{\mathbb V_h}^2
		=\norm{\bfv_h}_h^2+\norm{\theta_h}_h^2+\norm{\zeta_h}_h^2 .
	\]
	For $(\bfu_h,\widetilde\mu_h,\sigma_h)\in\mathbb V_h$,
	define
	\[
		\phi_h
		=
		\tfrac{2\dt}{3}\sigma_h
		+
		\tfrac{4\phi_h^n-\phi_h^{n-1}}{3}.
	\]
	Then $D_{2\tau}\phi_h=\sigma_h$.
	Moreover, since the previous two levels have mass \(m_0\) and
	\((\sigma_h,1)_h=0\),
	\[
		(\phi_h,1)_h=m_0 .
	\]
	Define $\phi_h^\dagger = \tfrac{3\phi_h-\phi_h^n}{2}$.

	For \(z_h=(\bfu_h,\widetilde\mu_h,\sigma_h)\in\mathbb V_h\) and
	for all $(\bfv_h,\psi_h,\eta_h)\in\mathbb V_h$, define
	\(\mathcal R(z_h)\in\mathbb V_h'\) by
	\[
		\begin{aligned}
			 & \left\langle
			\mathcal R(\bfu_h,\widetilde\mu_h,\sigma_h),
			(\bfv_h,\psi_h,\eta_h)
			\right\rangle                                \\
			 & \quad =
			(D_{2\tau}\bfu_h,\bfv_h)_h
			+\nu(D_h\bfu_h,D_h\bfv_h)_h
			+b_h(\bfu_h^{*,n+1},\bfu_h,\bfv_h)           \\
			 & \qquad
			+(\phi_h^{*,n+1}D_h\widetilde\mu_h,\bfv_h)_h \\
			 & \qquad
			+(\sigma_h,\psi_h)_h
			+(D_h\widetilde\mu_h,D_h\psi_h)_h
			+(d_h(\phi_h^{*,n+1}\bfu_h),\psi_h)_h        \\
			 & \qquad
			-(\widetilde\mu_h,\eta_h)_h
			+\eps^2(D_h\phi_h,D_h\eta_h)_h               \\
			 & \qquad
			+
			\bigl(
			\chi(\phi_h^\dagger,\phi_h^{\dagger,n})
			-\phi_h^{*,n+1},
			\eta_h
			\bigr)_h ,
		\end{aligned}
	\]
	Here $D_{2\tau}\bfu_h = \tfrac{3\bfu_h-4\bfu_h^n+\bfu_h^{n-1}}{2\dt}$. The map \(\mathcal R:\mathbb V_h\to\mathbb V_h'\) is continuous.

	Taking $(\bfv_h,\psi_h,\eta_h) = (\bfu_h,\widetilde\mu_h,\sigma_h)$, using Lemma~\ref{lem:mac-convection-form} and the summation-by-parts cancellation
	\[
		(\phi_h^{*,n+1}D_h\widetilde\mu_h,\bfu_h)_h
		+
		(d_h(\phi_h^{*,n+1}\bfu_h),\widetilde\mu_h)_h=0,
	\]
	we obtain
	\[
		\begin{aligned}
			 & \left\langle
			\mathcal R(\bfu_h,\widetilde\mu_h,\sigma_h),
			(\bfu_h,\widetilde\mu_h,\sigma_h)
			\right\rangle                  \\
			 & \quad =
			(D_{2\tau}\bfu_h,\bfu_h)_h
			+\nu\norm{D_h\bfu_h}_h^2
			+\norm{D_h\widetilde\mu_h}_h^2 \\
			 & \qquad
			+\eps^2(D_h\phi_h,D_h\sigma_h)_h
			+
			\bigl(
			\chi(\phi_h^\dagger,\phi_h^{\dagger,n})
			-\phi_h^{*,n+1},
			\sigma_h
			\bigr)_h .
		\end{aligned}
	\]
	Multiplying by \(\dt\) and applying Lemma~\ref{lem:bdf2-algebra} to
	\(\bfu_h\), \(D_h\phi_h\), and \(\phi_h\), we get
	\[
		\begin{aligned}
			 & \dt
			\left\langle
			\mathcal R(\bfu_h,\widetilde\mu_h,\sigma_h),
			(\bfu_h,\widetilde\mu_h,\sigma_h)
			\right\rangle           \\
			 & \quad \ge
			\calG(\bfu_h,\bfu_h^n)
			+
			\eps^2\calG(D_h\phi_h,D_h\phi_h^n)
			+
			(F(\phi_h^\dagger),1)_h \\
			 & \qquad
			+
			\nu\dt\norm{D_h\bfu_h}_h^2
			+
			\dt\norm{D_h\widetilde\mu_h}_h^2
			-C_n .
		\end{aligned}
	\]
	Here \(C_n\) depends only on the known levels \(n\) and \(n-1\).

	Since \(F\) has quartic growth and the mass of \(\phi_h^\dagger\) is fixed,
	the right-hand side tends to \(+\infty\) as
	\[
		\norm{\bfu_h}_h
		+
		\norm{\widetilde\mu_h}_h
		+
		\norm{\sigma_h}_h
		\to\infty .
	\]
	Indeed, \(\widetilde\mu_h\) and \(\sigma_h\) have zero mean, so the discrete
	Poincar\'e inequality controls them by their discrete gradients, and
	\(\phi_h^\dagger\) is an affine function of \(\sigma_h\). Therefore there
	exists \(R>0\) such that
	\[
		\left\langle
		\mathcal R(\bfu_h,\widetilde\mu_h,\sigma_h),
		(\bfu_h,\widetilde\mu_h,\sigma_h)
		\right\rangle
		\ge0
	\]
	whenever $\norm{\bfu_h}_h^2 + \norm{\widetilde\mu_h}_h^2 + \norm{\sigma_h}_h^2 =R^2$.
	By Lemma~\ref{lem:brouwer}, there exists
	\[
		(\bfu_h^{n+1},\widetilde\mu_h^{n+1},\sigma_h^{n+1})
		\in\mathbb V_h
	\]
	such that
	\[
		\mathcal R(\bfu_h^{n+1},
		\widetilde\mu_h^{n+1},
		\sigma_h^{n+1})=0 .
	\]
	Set
	\[
		\phi_h^{n+1}
		=
		\tfrac{2\dt}{3}\sigma_h^{n+1}
		+
		\tfrac{4\phi_h^n-\phi_h^{n-1}}{3}.
	\]
	Then $D_{2\tau}\phi_h^{n+1}=\sigma_h^{n+1}, \ (\phi_h^{n+1},1)_h=m_0$.

	It remains to recover the mean of the chemical potential. Define
	\[
		g_h^{n+1}
		=
		\eps^2A_h\phi_h^{n+1}
		+
		\chi(\phi_h^{\dagger,n+1},\phi_h^{\dagger,n})
		-
		\phi_h^{*,n+1}.
	\]
	Choose $\overline\mu_h^{\,n+1} = \tfrac{(g_h^{n+1},1)_h}{(1,1)_h}$,
	and set $\mu_h^{n+1} = \widetilde\mu_h^{n+1} + \overline\mu_h^{\,n+1}$.
	Since the residual equation in the \(\eta_h\)-component holds for every
	\(\eta_h\in\mathsf C_{h,0}\), we have
	\[
		\mu_h^{n+1}
		-\eps^2A_h\phi_h^{n+1}
		-\chi(\phi_h^{\dagger,n+1},\phi_h^{\dagger,n})
		+\phi_h^{*,n+1}
		=0 .
	\]
	Thus \eqref{eq:scheme-mu} holds. The phase equation holds for all
	cell-centered test functions, because both sides have zero mean on constants. We next recover the pressure. Define
	\[
		\begin{aligned}
			\mathcal L_h(\bfv_h)
			:={} &
			-(D_{2\tau}\bfu_h^{n+1},\bfv_h)_h
			-\nu(D_h\bfu_h^{n+1},D_h\bfv_h)_h               \\
			     & -b_h(\bfu_h^{*,n+1},\bfu_h^{n+1},\bfv_h)
			-(\phi_h^{*,n+1}D_h\mu_h^{n+1},\bfv_h)_h .
		\end{aligned}
	\]
	The velocity equation has been solved on \(\mathsf U_h\), hence $\mathcal L_h(\bfv_h)=0
		\ \forall\,\bfv_h\in\mathsf E_h \text{with} d_h\bfv_h=0$.
	Lemma~\ref{lem:mac-infsup} gives a unique zero-mean pressure
	\(p_h^{n+1}\) such that
	\[
		(D_hp_h^{n+1},\bfv_h)_h
		=
		\mathcal L_h(\bfv_h)
		\qquad
		\forall\,\bfv_h\in\mathsf E_h .
	\]
	Thus a solution of \eqref{eq:scheme} exists.

	We now prove uniqueness. Let two solutions with the same history and the
	same pressure normalization be given. Write
	\begin{equation*}
		\widehat\bfu
		=
		\bfu_h^{(1)}-\bfu_h^{(2)},
		\
		\widehat\phi
		=
		\phi_h^{(1)}-\phi_h^{(2)}, \
		\widehat\mu
		=
		\mu_h^{(1)}-\mu_h^{(2)},
		\
		\widehat p
		=
		p_h^{(1)}-p_h^{(2)} .
	\end{equation*}
	By mass conservation and the pressure normalization, $(\widehat\phi,1)_h=0, \ (\widehat p,1)_h=0$. Subtracting the two systems and testing the momentum equation by
	\(\widehat\bfu\), the phase equation by \(\widehat\mu\), and the chemical
	potential equation by $D_{2\tau}\widehat\phi^{n+1} = \tfrac{3\widehat\phi}{2\dt}$,
	we obtain, after the usual cancellations,
	\[
		\begin{aligned}
			 & \tfrac{3}{2\dt}\norm{\widehat\bfu}_h^2
			+\nu\norm{D_h\widehat\bfu}_h^2
			+\tfrac{3\eps^2}{2\dt}\norm{D_h\widehat\phi}_h^2
			+\norm{D_h\widehat\mu}_h^2                \\
			 & \quad
			+
			\tfrac{3}{2\dt}
			\left(
			\chi(\phi_h^{(1),\dagger,n+1},\phi_h^{\dagger,n})
			-\chi(\phi_h^{(2),\dagger,n+1},\phi_h^{\dagger,n}),
			\widehat\phi
			\right)_h
			=0 .
		\end{aligned}
	\]
	Since $\widehat\phi = \frac23 \left(\phi_h^{(1),\dagger,n+1} - \phi_h^{(2),\dagger,n+1} \right)$, Lemma~\ref{lem:chi-estimates} shows that the last term is nonnegative. Therefore
	\[
		\widehat\bfu=0,
		\quad
		D_h\widehat\phi=0,
		\quad
		D_h\widehat\mu=0 .
	\]
	Since \((\widehat\phi,1)_h=0\), the discrete Poincar\'e inequality gives $\widehat\phi=0$.
	The chemical potential equation then gives $\widehat\mu=0$. Finally, the momentum equation reduces to
	\[
		(D_h\widehat p,\bfv_h)_h=0
		\quad
		\forall\,\bfv_h\in\mathsf E_h .
	\]
	Lemma~\ref{lem:mac-infsup} and \((\widehat p,1)_h=0\) imply $\widehat p=0$. The solution is unique.
\end{proof}

\begin{theorem}\label{thm:energy-stability}
	Any solution of \eqref{eq:scheme} satisfies the following discrete energy law,
	for \(1\le n\le N-1\):
	\begin{align}
		 & \calE_h^{n+1}-\calE_h^n
		+\dt\norm{D_h\mu_h^{n+1}}_h^2
		+\nu\dt\norm{D_h\bfu_h^{n+1}}_h^2 \notag \\
		 & \quad
		+\tfrac14\norm{\delta^2\bfu_h^{n+1}}_h^2
		+\tfrac{\eps^2}{4}\norm{D_h\delta^2\phi_h^{n+1}}_h^2
		+\tfrac34\norm{\delta^2\phi_h^{n+1}}_h^2=0.
		\label{eq:energy-law}
	\end{align}
	Here, for \(j\ge1\),
	\begin{align}\label{eq:discrete-energy}
		\calE_h^j
		 & =\calG(\bfu_h^j,\bfu_h^{j-1})
		+\eps^2\calG(D_h\phi_h^j,D_h\phi_h^{j-1})
		+(F(\phi_h^{\dagger,j}),1)_h
		+\tfrac38\norm{\phi_h^j-\phi_h^{j-1}}_h^2.
	\end{align}
	In particular, \(\calE_h^{n+1}\le \calE_h^n\).
\end{theorem}

\begin{proof}
	Test \eqref{eq:scheme-u} with \(\dt\bfu_h^{n+1}\) and
	\eqref{eq:scheme-phi} with \(\dt\mu_h^{n+1}\). By
	Lemma~\ref{lem:mac-sbp-identities}, \eqref{eq:scheme-div}, and Lemma~\ref{lem:mac-convection-form},
	\[
		(D_hp_h^{n+1},\bfu_h^{n+1})_h
		=-(p_h^{n+1},d_h\bfu_h^{n+1})_h=0,
		\quad
		b_h(\bfu_h^{*,n+1},\bfu_h^{n+1},\bfu_h^{n+1})=0,
	\]
	where \(d_h\bfu_h^{*,n+1}=0\) follows from the divergence-free history. The
	phase-transport and capillary terms cancel by Lemma~\ref{lem:mac-sbp-identities}:
	\[
		(\phi_h^{*,n+1}D_h\mu_h^{n+1},\bfu_h^{n+1})_h
		+(d_h(\phi_h^{*,n+1}\bfu_h^{n+1}),\mu_h^{n+1})_h=0.
	\]
	Adding the two tested equations gives
	\begin{align}\label{eq:energy-tested-balance}
		 & \dt(D_{2\tau}\bfu_h^{n+1},\bfu_h^{n+1})_h
		+\nu\dt\norm{D_h\bfu_h^{n+1}}_h^2
		\notag                                       \\
		 & \quad
		+\dt(D_{2\tau}\phi_h^{n+1},\mu_h^{n+1})_h
		+\dt\norm{D_h\mu_h^{n+1}}_h^2=0 .
	\end{align}

	Next test \eqref{eq:scheme-mu} with \(\dt D_{2\tau}\phi_h^{n+1}\):
	\begin{align}\label{eq:energy-phase-rewrite}
		\dt(D_{2\tau}\phi_h^{n+1},\mu_h^{n+1})_h
		 & =\eps^2\dt(A_h\phi_h^{n+1},D_{2\tau}\phi_h^{n+1})_h \notag \\
		 & \quad+\dt(\chi(\phi_h^{\dagger,n+1},\phi_h^{\dagger,n})
		-\phi_h^{*,n+1},D_{2\tau}\phi_h^{n+1})_h .
	\end{align}
	Apply Lemma~\ref{lem:bdf2-algebra} to the velocity term in
	\eqref{eq:energy-tested-balance} and, after Lemma~\ref{lem:mac-sbp-identities}, to the gradient
	term in \eqref{eq:energy-phase-rewrite}. The potential term is given by
	Lemma~\ref{lem:bdf2-algebra}, and substitution gives \eqref{eq:energy-law}.
\end{proof}

\section{Convergence analysis}

For \(\psi,q\in C(\overline\Omega)\),
\(\bfv=(v_1,v_2)\in C(\overline\Omega)^2\), and the phase field \(\phi(t)\),
define
\begin{equation}
	\label{eq:restriction-operators}
	\begin{aligned}
		(R_c\psi)_{i+\frac12,j+\frac12}
		 & =\psi(x_{i+\frac12},y_{j+\frac12}), \\
		(R_e\bfv)_{1,i,j+\frac12}
		 & =v_1(x_i,y_{j+\frac12}),            \\
		(R_e\bfv)_{2,i+\frac12,j}
		 & =v_2(x_{i+\frac12},y_j),            \\
		R_p q
		 & =R_c q-\overline{R_c q},            \\
		\overline{\eta_h}
		 & =|\Omega|^{-1}(\eta_h,1)_h,         \\
		I_c\phi(t)
		 & =R_c\phi(t)
		+\overline{R_c\phi(t_0)}
		-\overline{R_c\phi(t)} .
	\end{aligned}
\end{equation}
Here \(R_e\bfv\) is restricted to the $\mathsf E_h$. By construction we have $D_hI_c\phi=D_hR_c\phi, \ A_hI_c\phi=A_hR_c\phi$.
If $\phi\in L^\infty(0,T;W^{2,\infty}(\Omega))$, then
\[
	\norm{I_c\phi(t)-R_c\phi(t)}_h
	\le
	Ch^2
	\norm{\phi}_{L^\infty(0,T;W^{2,\infty})}.
\]

Define the velocity projection by
\begin{equation}\label{eq:velocity-projection-def}
	\Pi_u\bfv
	=
	R_e\bfv+D_h\theta_{\bfv},
\end{equation}
where \(\theta_{\bfv}\in\mathsf C_{h,0}\) solves
\begin{equation}\label{eq:velocity-projection-poisson}
	(D_h\theta_{\bfv},D_h\eta_h)_h
	=
	(d_hR_e\bfv-R_c(\nabla\!\cdot\!\bfv),\eta_h)_h,
	\quad
	\forall\,\eta_h\in\mathsf C_{h,0}.
\end{equation}
We assume the following regularity and start-up estimates.

\begin{assumption}\label{ass:smooth-solution}
	The exact solution of
	\eqref{eq:continuous-chns}--\eqref{eq:continuous-bc}
	satisfies
	\[
		\phi\in W^{3,\infty}(0,T;W^{2,\infty})\cap W^{2,\infty}(0,T;W^{5,\infty}),
		\quad
		\mu\in W^{2,\infty}(0,T;W^{3,\infty}),
	\]
	\[
		\bfu\in W^{3,\infty}(0,T;L^2)^2\cap
		W^{2,\infty}(0,T;W^{3,\infty})^2,
		\quad
		p\in W^{2,\infty}(0,T;W^{2,\infty}).
	\]
\end{assumption}

\begin{assumption}\label{ass:start-up}
	The initial BDF2 energy satisfies $\calE_h^1\le C$.
	Moreover,
	\begin{align*}
		 & (\phi_h^j,1)_h=(I_c\phi(t_j),1)_h,
		\qquad j=0,1,                           \\
		 & \norm{I_c\phi(t_j)-\phi_h^j}_h
		+\norm{D_h(I_c\phi(t_j)-\phi_h^j)}_h
		+\norm{\Pi_u\bfu(t_j)-\bfu_h^j}_h       \\
		 & \quad
		+\dt^{1/2}\norm{D_h(\Pi_u\bfu(t_j)-\bfu_h^j)}_h
		+\norm{R_c\mu(t_j)-\mu_h^j}_h
		+\norm{D_h(R_c\mu(t_j)-\mu_h^j)}_h      \\
		 & \quad
		+\norm{R_p p(t_j)-p_h^j}_h
		\le C(\dt^2+h^2),
		\qquad j=0,1,                           \\
		 & \norm{\mu_h^j}_h+\norm{D_h\mu_h^j}_h
		+\norm{A_h\phi_h^j}_h+\norm{D_h\bfu_h^j}_h
		\le C,
		\qquad j=0,1,                           \\
		 & \norm{\phi_h^1-\phi_h^0}_h
		+\norm{A_h(\phi_h^1-\phi_h^0)}_h
		+\norm{\mu_h^1-\mu_h^0}_h
		\le C\dt .
	\end{align*}
	The first-step residuals obtained from the projected exact data satisfy the
	same \(O(\dt^2+h^2)\) bounds in the corresponding norms.
\end{assumption}

\begin{lemma}[\cite{RuiLi2017,LiShen2020NS,LiShen2020}]
	\label{lem:mac-projections}
	Let \(\psi\in W^{3,\infty}(\Omega)\), \(\zeta\in W^{5,\infty}(\Omega)\),
	\(\bfv\in W^{3,\infty}(\Omega)^2\), and \(q\in W^{2,\infty}(\Omega)\), where
	\(q\) has zero continuous mean. Then
	\begin{equation*}
		\begin{aligned}
			\norm{D_hR_c\psi-R_e\nabla\psi}_h
			 & \le Ch^2\norm{\psi}_{W^{3,\infty}},  \\
			\norm{A_hR_c\zeta-R_c(-\Delta\zeta)}_h
			+\norm{D_h\bigl(A_hR_c\zeta-R_c(-\Delta\zeta)\bigr)}_h
			 & \le Ch^2\norm{\zeta}_{W^{5,\infty}}, \\
			\norm{\Pi_u\bfv-R_e\bfv}_h
			+h\norm{D_h(\Pi_u\bfv-R_e\bfv)}_h
			 & \le Ch^2\norm{\bfv}_{W^{3,\infty}},  \\
			\norm{R_p q-R_c q}_h
			 & \le Ch^2\norm{q}_{W^{2,\infty}}.
		\end{aligned}
	\end{equation*}
	If \(d_hR_e\bfv\) and \(R_c(\nabla\cdot\bfv)\) have the same discrete mean,
	then
	\[
		d_h\Pi_u\bfv=R_c(\nabla\cdot\bfv),
	\]
	in particular \(d_h\Pi_u\bfu(t)=0\) for the incompressible exact velocity.
\end{lemma}

\begin{proof}
	The first two estimates are standard centered-difference Taylor estimates.
	Since
	\[
		R_pq-R_cq=-\overline{R_cq},
	\]
	the last estimate follows from midpoint quadrature and the zero-mean
	condition on \(q\).

	Set $\delta_h=d_hR_e\bfv-R_c(\nabla\!\cdot\!\bfv)$, a Taylor expansion gives $\norm{\delta_h}_h
		\le Ch^2\norm{\bfv}_{W^{3,\infty}}$.
	Taking \(\eta_h=\theta_{\bfv}\) in
	\eqref{eq:velocity-projection-poisson} and using the discrete
	Poincar\'e inequality yields
	\[
		\norm{D_h\theta_{\bfv}}_h
		\le C\norm{\delta_h}_h
		\le Ch^2\norm{\bfv}_{W^{3,\infty}}.
	\]
	Since $\Pi_u\bfv-R_e\bfv=D_h\theta_{\bfv}$,
	the first part of the projection estimate follows. The second follows from
	the  inverse inequality,
	\[
		h\norm{D_hD_h\theta_{\bfv}}_h
		\le C\norm{D_h\theta_{\bfv}}_h.
	\]

	Finally, for any \(\eta_h\in\mathsf C_{h,0}\),
	\[
		(d_h\Pi_u\bfv-R_c(\nabla\!\cdot\!\bfv),\eta_h)_h
		=(\delta_h,\eta_h)_h-(D_h\theta_{\bfv},D_h\eta_h)_h=0.
	\]
	Hence \(d_h\Pi_u\bfv-R_c(\nabla\!\cdot\!\bfv)\) is constant, and the mean
	compatibility implies
	\[
		d_h\Pi_u\bfv=R_c(\nabla\!\cdot\!\bfv).
	\]
\end{proof}

\subsection{A priori estimates}

We use the following discrete Gronwall inequalities.

\begin{lemma}[Discrete Gronwall, \cite{ShenXu2018}]
	\label{lem:discrete-gronwall}
	Suppose nonnegative sequences \(a_m,b_m,c_m\) satisfy
	$\dt\sum_{m=1}^{N}c_m\le C_1$
	and
	\[
		a_\ell+\dt\sum_{m=1}^{\ell}b_m
		\le C_2+\dt\sum_{m=1}^{\ell-1}c_m a_m,
		\qquad 1\le \ell\le N.
	\]
	Then
	\[
		a_\ell+\dt\sum_{m=1}^{\ell}b_m
		\le C_2e^{C_1},
		\qquad 1\le \ell\le N.
	\]
\end{lemma}
\begin{lemma}[Weighted discrete Gronwall]
	\label{lem:weighted-gronwall}
	Let \(0\le\alpha<1\). Suppose nonnegative sequences \(a_m,b_m,c_m\) satisfy
	$\dt\sum_{m=1}^{N}c_m\le C_1$,
	and
	\[
		a_\ell+\dt\sum_{m=1}^{\ell}b_m
		\le C_2+\dt\sum_{m=1}^{\ell-1}c_m
		\sum_{j=1}^{m}\alpha^{m-j}a_j,
		\qquad 1\le \ell\le N.
	\]
	Then
	\[
		a_\ell+\dt\sum_{m=1}^{\ell}b_m
		\le
		\frac{C_2}{1-\alpha}
		\exp\!\left(\frac{C_1}{1-\alpha}\right),
		\qquad 1\le \ell\le N.
	\]
\end{lemma}
The proof of Lemma~\ref{lem:weighted-gronwall} is given in
Appendix~\ref{app:gronwall-proofs}.

\begin{lemma}\label{lem:base-consequences}
	Under Assumption~\ref{ass:start-up},
	\[
		\max_{1\le n\le N}
		\Bigl(
		\norm{\bfu_h^n}_h
		+\norm{D_h\phi_h^n}_h
		+\norm{\phi_h^{\dagger,n}}_{\ell^4}^2
		\Bigr)
		+
		\Bigl(
		\dt\sum_{n=1}^{N}\norm{D_h\mu_h^n}_h^2
		\Bigr)^{1/2}
		+
		\Bigl(
		\nu\dt\sum_{n=1}^{N}\norm{D_h\bfu_h^n}_h^2
		\Bigr)^{1/2}
		\le C .
	\]
\end{lemma}

\begin{proof}
	Summing \eqref{eq:energy-law} from \(n=1\) to \(m-1\), \(2\le m\le N\),
	and discarding nonnegative terms yields
	\[
		\max_{1\le n\le m}
		\Bigl(
		\norm{\bfu_h^n}_h^2
		+\norm{D_h\phi_h^n}_h^2
		+(F(\phi_h^{\dagger,n}),1)_h
		+\norm{\phi_h^n-\phi_h^{n-1}}_h^2
		\Bigr)
		+\dt\sum_{n=2}^{m}
		\Bigl(
		\norm{D_h\mu_h^n}_h^2
		+\nu\norm{D_h\bfu_h^n}_h^2
		\Bigr)
		\le C\calE_h^1 .
	\]
	Here we also used the equivalence of \(\calG\) and the two-level
	\(L^2\) norm. Assumption~\ref{ass:start-up} bounds \(\calE_h^1\) and
	provides the missing \(n=1\) dissipation terms, yielding the estimates for
	\(\bfu_h^n\), \(D_h\phi_h^n\), \(D_h\mu_h^n\), and \(D_h\bfu_h^n\).
	Furthermore,
	\(F(s)=\frac14s^4-\frac12s^2+\frac14\ge \frac18s^4-C\), so
	\((F(\phi_h^{\dagger,n}),1)_h\le C\) implies
	\(\norm{\phi_h^{\dagger,n}}_{\ell^4}^4\le C\).
\end{proof}

\begin{lemma}\label{lem:linfty-phase}
	Under Assumption~\ref{ass:start-up}, for each fixed \(T=N\dt\),
	\[
		\max_{1\le n\le N}\Bigl(
		\norm{A_h\phi_h^n}_h^2+\norm{\phi_h^n}_{\ell^\infty}^2
		+\norm{\mu_h^n}_h^2\Bigr)
		+\dt\sum_{n=1}^N\norm{D_{2\tau}\phi_h^n}_{H_h^{-1}}^2
		\le C_T .
	\]
	The constant \(C_T\) is independent of \(h\) and \(\dt\).
\end{lemma}

\begin{proof}
	It suffices first to prove the stronger bound
	\begin{equation}\label{eq:strong-phase-time-bound-target}
		\max_{1\le n\le N}\norm{\mu_h^n}_h^2
		+\dt\sum_{n=1}^N\norm{D_{2\tau}\phi_h^n}_h^2\le C_T .
	\end{equation}
	The asserted \(H_h^{-1}\) estimate then follows at the end from the phase
	equation. We use two identities following from the
	shifted BDF2 quantities in \eqref{eq:bdf2-defs}:
	\begin{equation}\label{eq:shifted-bdf2-increment}
		\phi_h^{\dagger,n+1}-\phi_h^{\dagger,n}=
		\dt D_{2\tau}\phi_h^{n+1},
		\quad
		q^{n+1}=\tfrac23 q^{\dagger,n+1}+\tfrac13 q^n
	\end{equation}
	for any sequence \(q^n\). Iterating the second identity gives
	\[
		\begin{aligned}
			q^{n+1}
			 & =\tfrac23 q^{\dagger,n+1}+\tfrac13 q^n
			=\tfrac23 q^{\dagger,n+1}
			+\tfrac23\left(\tfrac13\right)q^{\dagger,n}
			+\left(\tfrac13\right)^2q^{n-1}           \\
			 & =\cdots
			=\tfrac23\sum_{j=1}^{n+1}\left(\tfrac13\right)^{n+1-j}q^{\dagger,j}
			+\left(\tfrac13\right)^{n+1}q^0 .
		\end{aligned}
	\]
	Consequently,
	\begin{equation}\label{eq:shifted-bdf2-history}
		\begin{aligned}
			\norm{q^{n+1}}_h^2 & \leq \tfrac43 \norm{\sum_{j=1}^{n+1}\left(\tfrac13\right)^{n+1-j}q^{\dagger,j}}^2
			+2 \left(\tfrac13\right)^{2n+2} \norm{q^0}^2                                                           \\
			                   & = \left\|\sum_{j=1}^{n+1}\left(\tfrac13\right)^{(n+1-j)/2}
			\left[\left(\tfrac13\right)^{(n+1-j)/2}q^{\dagger,j}\right]
			\right\|_h^2 +2 \left(\tfrac13\right)^{2n+2} \norm{q^0}^2                                              \\
			                   & \le C\sum_{j=1}^{n+1}\left(\tfrac13\right)^{n+1-j}
			\norm{q^{\dagger,j}}_h^2+C\left(\tfrac13\right)^n\norm{q^0}_h^2 .
		\end{aligned}
	\end{equation}

	Define $G_h^k=\mu_h^k-\eps^2A_h\phi_h^k \ (k \geq 1)$. For \(k\ge2\), \eqref{eq:scheme-mu} implies
	\[
		G_h^k=\chi(\phi_h^{\dagger,k},\phi_h^{\dagger,k-1})-
		\phi_h^{*,k} .
	\]
	Hence,
	\begin{equation}\label{eq:first-G-difference}
		\mu_h^2-\mu_h^1=\eps^2A_h(\phi_h^2-\phi_h^1)+G_h^2-G_h^1.
	\end{equation}
	Testing \eqref{eq:first-G-difference} by \(\mu_h^2\) and using
	\((a-b,a)_h\ge \tfrac12\norm{a}_h^2-\tfrac12\norm{b}_h^2\) gives
	\begin{equation}\label{eq:first-G-tested}
		\begin{aligned}
			\frac12\norm{\mu_h^2}_h^2
			 & \le
			C
			+\eps^2(A_h(\phi_h^2-\phi_h^1),\mu_h^2)_h
			+(G_h^2-G_h^1,\mu_h^2)_h
			\\
			 & =
			C
			+\tfrac{2\eps^2\dt}{3}
			(A_hD_{2\tau}\phi_h^2,\mu_h^2)_h
			+\tfrac{\eps^2}{3}
			(A_h(\phi_h^1-\phi_h^0),\mu_h^2)_h
			+(G_h^2-G_h^1,\mu_h^2)_h .
		\end{aligned}
	\end{equation}
	At $n = 2$, \eqref{eq:scheme-phi} gives
	\[
		D_{2\tau}\phi_h^2+A_h\mu_h^2+d_h(\phi_h^{*,2}\bfu_h^2)=0.
	\]
	Testing by \(D_{2\tau}\phi_h^2\) and using Lemma~\ref{lem:mac-sbp-identities},
	\begin{equation}\label{eq:first-phase-pairing}
		(A_hD_{2\tau}\phi_h^2,\mu_h^2)_h
		=-\norm{D_{2\tau}\phi_h^2}_h^2
		-(d_h(\phi_h^{*,2}\bfu_h^2),D_{2\tau}\phi_h^2)_h .
	\end{equation}
	Substituting \eqref{eq:first-phase-pairing} into
	\eqref{eq:first-G-tested}, moving
	\(\norm{D_{2\tau}\phi_h^2}_h^2\) to the left-hand side, and using
	\(d_h\bfu_h^2=0\),
	\eqref{eq:mac-product-def},
	Lemma~\ref{lem:mac-sbp-identities},
	Lemma~\ref{lem:mac-discrete-estimates} with \(p=6\),
	\eqref{eq:mass-conservation},
	Lemma~\ref{lem:base-consequences},
	Assumption~\ref{ass:start-up},
	and Young's inequality, we obtain
	\begin{align*}
		 & \tfrac12\norm{\mu_h^2}_h^2+\tfrac{2\eps^2\dt}{3}
		\norm{D_{2\tau}\phi_h^2}_h^2                                                                            \\
		 & \quad \le C+\tfrac{2\eps^2\dt}{3}
		\left|(d_h(\phi_h^{*,2}\bfu_h^2),D_{2\tau}\phi_h^2)_h\right|
		+\tfrac{\eps^2}{3}\left|(A_h(\phi_h^1-\phi_h^0),\mu_h^2)_h\right| +\left|(G_h^2-G_h^1,\mu_h^2)_h\right| \\
		 & \quad \le C+C\dt\norm{\bfu_h^2}_{\ell^4}
		\norm{D_h\phi_h^{*,2}}_{\ell^4}\norm{D_{2\tau}\phi_h^2}_h
		+C\norm{A_h(\phi_h^1-\phi_h^0)}_h
		\norm{\mu_h^2}_h                                                                                        \\
		 & \quad \quad +\norm{\chi(\phi_h^{\dagger,2},\phi_h^{\dagger,1})-
			                \phi_h^{*,2}}_h
		\norm{\mu_h^2}_h
		+\norm{\mu_h^1-\eps^2A_h\phi_h^1}_h\norm{\mu_h^2}_h
		\displaybreak[0]                                                                                        \\
		 & \quad \le C+C\dt\norm{\bfu_h^2}_h^{1/2}\norm{D_h\bfu_h^2}_h^{1/2}
		\norm{D_h\phi_h^{*,2}}_h^{1/2}\norm{A_h\phi_h^{*,2}}_h^{1/2}
		\norm{D_{2\tau}\phi_h^2}_h+C\dt\norm{\mu_h^2}_h                                                         \\
		 & \quad \quad +C\left(1\mathrel{}+\norm{\phi_h^{*,2}}_h\right)
		\norm{\mu_h^2}_h +C\left(\norm{\mu_h^1}_h+ \norm{A_h\phi_h^1}_h\right)\norm{\mu_h^2}_h                  \\
		 & \quad \le C+C\dt(1+\norm{D_h\bfu_h^2}_h)
		(1+\norm{A_h\phi_h^{*,2}}_h)
		\norm{D_{2\tau}\phi_h^2}_h+C(1+\dt)\norm{\mu_h^2}_h                                                     \\
		 & \quad \le C+C\dt(1+\norm{D_h\bfu_h^2}_h)
		\norm{D_{2\tau}\phi_h^2}_h+C(1+\dt)\norm{\mu_h^2}_h                                                     \\
		 & \quad \le C+C\dt(1+\norm{D_h\bfu_h^2}_h^2)
		+\tfrac{\eps^2\dt}{6}\norm{D_{2\tau}\phi_h^2}_h^2
		+\tfrac18\norm{\mu_h^2}_h^2                                                                             \\
		 & \quad \le C_T+\tfrac{\eps^2\dt}{6}\norm{D_{2\tau}\phi_h^2}_h^2
		+\tfrac18\norm{\mu_h^2}_h^2 .
	\end{align*}
	Here \(\norm{A_h\phi_h^{*,2}}_h\) is bounded by
	Assumption~\ref{ass:start-up}. The cubic defect is estimated by
	Lemmas~\ref{lem:chi-estimates} and~\ref{lem:base-consequences}, and
	Assumption~\ref{ass:start-up}. Moreover,
	\(\dt\norm{D_h\bfu_h^2}_h^2\le C_T\) by Lemma~\ref{lem:base-consequences}.
	Moving the last two terms to the left gives
	\[
		\norm{\mu_h^2}_h^2+
		\dt\norm{D_{2\tau}\phi_h^2}_h^2\le C_T .
	\]
	Also,
	\[
		\norm{\mu_h^{\dagger,2}}_h^2
		\le C(\norm{\mu_h^2}_h^2+\norm{\mu_h^1}_h^2)
		\le C_T,
	\]
	by the preceding bound and Assumption~\ref{ass:start-up}. Hence
	\begin{equation}\label{eq:first-bdf2-highnorm}
		\norm{\mu_h^2}_h^2+\norm{\mu_h^{\dagger,2}}_h^2
		+\dt\norm{D_{2\tau}\phi_h^2}_h^2\le C_T .
	\end{equation}
	Applying the shifted BDF2 combination to
	\(\mu_h^k=\eps^2A_h\phi_h^k+G_h^k\) gives, for \(n\ge2\),
	\begin{equation}\label{eq:shifted-chemical-diff}
		\mu_h^{\dagger,n+1}-\mu_h^{\dagger,n}
		=\eps^2\dt A_hD_{2\tau}\phi_h^{n+1}
		+\tfrac32(G_h^{n+1}-G_h^n)-\tfrac12(G_h^n-G_h^{n-1}).
	\end{equation}
	The same shifted combination applied to \eqref{eq:scheme-phi} yields
	\begin{equation}\label{eq:shifted-phase-combination}
		\tfrac32D_{2\tau}\phi_h^{n+1}-\tfrac12D_{2\tau}\phi_h^n
		+A_h\mu_h^{\dagger,n+1}
		+\tfrac32d_h(\phi_h^{*,n+1}\bfu_h^{n+1})
		-\tfrac12d_h(\phi_h^{*,n}\bfu_h^n)=0.
	\end{equation}
	Testing \eqref{eq:shifted-chemical-diff} by
	\(\mu_h^{\dagger,n+1}\), summation-by-parts, and substituting
	\eqref{eq:shifted-phase-combination} into the obtained formulation gives
	\begin{equation}\label{eq:mu-shifted-start}
		\begin{aligned}
			 & \tfrac12\left(\norm{\mu_h^{\dagger,n+1}}_h^2-
			\norm{\mu_h^{\dagger,n}}_h^2\right)
			+\tfrac{\eps^2\dt}{4}\left(\norm{D_{2\tau}\phi_h^{n+1}}_h^2-
			\norm{D_{2\tau}\phi_h^n}_h^2\right) +\eps^2\dt\norm{D_{2\tau}\phi_h^{n+1}}_h^2 \\
			\le
			 & -\eps^2\dt(\tfrac32d_h(\phi_h^{*,n+1}\bfu_h^{n+1})
			-\tfrac12d_h(\phi_h^{*,n}\bfu_h^n),D_{2\tau}\phi_h^{n+1})_h
			+(\tfrac32(G_h^{n+1}-G_h^n)-\tfrac12(G_h^n-G_h^{n-1}),
			\mu_h^{\dagger,n+1})_h,
		\end{aligned}
	\end{equation}
	where the following equation is utilized in the above derivation
	\[
		\begin{aligned}
			 & \left(\tfrac32D_{2\tau}\phi_h^{n+1}-\tfrac12D_{2\tau}\phi_h^n,
			D_{2\tau}\phi_h^{n+1}\right)_h                                    \\
			 & \quad=\tfrac14\left(\norm{D_{2\tau}\phi_h^{n+1}}_h^2-
			\norm{D_{2\tau}\phi_h^n}_h^2\right)
			+\tfrac14\norm{D_{2\tau}\phi_h^{n+1}-D_{2\tau}\phi_h^n}_h^2
			+\norm{D_{2\tau}\phi_h^{n+1}}_h^2.
		\end{aligned}
	\]

	For \(k=n,n+1\), summation by parts, H\"older's inequality, and
	Lemma~\ref{lem:mac-discrete-estimates} give
	\begin{equation}\label{eq:transport-local-bound}
		\begin{aligned}
			 & \left|(d_h(\phi_h^{*,k}\bfu_h^k),
			D_{2\tau}\phi_h^{n+1})_h\right|                                \\
			 & \quad\le C\norm{\bfu_h^k}_{\ell^4}
			\norm{D_h\phi_h^{*,k}}_{\ell^4}
			\norm{D_{2\tau}\phi_h^{n+1}}_h                                 \\
			 & \quad\le C\norm{\bfu_h^k}_h^{1/2}\norm{D_h\bfu_h^k}_h^{1/2}
			\norm{D_h\phi_h^{*,k}}_h^{1/2}
			\norm{A_h\phi_h^{*,k}}_h^{1/2}
			\norm{D_{2\tau}\phi_h^{n+1}}_h                                 \\
			 & \quad\le C\left(1+\norm{D_h\bfu_h^k}_h\right)
			\left(1+\norm{A_h\phi_h^{*,k}}_h\right)
			\norm{D_{2\tau}\phi_h^{n+1}}_h .
		\end{aligned}
	\end{equation}
	Here Lemma~\ref{lem:base-consequences} controls \(\norm{\bfu_h^k}_h\) and
	\(\norm{D_h\phi_h^{*,k}}_h\). In the \(p=4\) estimate,
	Lemma~\ref{lem:mac-discrete-estimates} is applied to
	\(\phi_h^{*,k}-\overline{\phi_h^{*,k}}\), since both \(D_h\) and \(A_h\)
	annihilate constants. For \(k\ge2\), \eqref{eq:scheme-mu} gives
	\begin{equation}\label{eq:Ahphi-by-mu}
		\eps^2\norm{A_h\phi_h^k}_h
		\le \norm{\mu_h^k}_h
		+\norm{\chi(\phi_h^{\dagger,k},\phi_h^{\dagger,k-1})-\phi_h^{*,k}}_h
		\le C\left(1+\norm{\mu_h^k}_h\right),
	\end{equation}
	where the last inequality follows from Lemma~\ref{lem:chi-estimates},
	Lemma~\ref{lem:mac-discrete-estimates} with \(p=6\),
	\eqref{eq:mass-conservation}, Lemma~\ref{lem:base-consequences}, and
	Assumption~\ref{ass:start-up}. Since \(\phi_h^{*,k}=2\phi_h^{k-1}-\phi_h^{k-2}\),
	\eqref{eq:Ahphi-by-mu} yields
	\[
		\norm{A_h\phi_h^{*,k}}_h^2
		\le C\left(1+\norm{\mu_h^{k-1}}_h^2+\norm{\mu_h^{k-2}}_h^2\right).
	\]
	For \(k=n,n+1\), \eqref{eq:shifted-bdf2-history} with \(q=\mu_h\) gives
	\begin{equation}\label{eq:Aphi-star-history}
		\norm{A_h\phi_h^{*,k}}_h^2
		\le C\left(1+
		\sum_{j=1}^{n}\left(\tfrac13\right)^{n-j}
		\norm{\mu_h^{\dagger,j}}_h^2\right),
	\end{equation}
	with the start-up contribution absorbed into \(C\). Consequently, the transport
	term in \eqref{eq:mu-shifted-start} satisfies
	\begin{equation}\label{eq:transport-highnorm-estimate}
		\begin{aligned}
			 & \eps^2\dt\left|\left(\tfrac32d_h(\phi_h^{*,n+1}\bfu_h^{n+1})
			-\tfrac12d_h(\phi_h^{*,n}\bfu_h^n),D_{2\tau}\phi_h^{n+1}\right)_h\right| \\
			 & \quad\le \tfrac{\eps^2\dt}{8}\norm{D_{2\tau}\phi_h^{n+1}}_h^2
			+C\dt\left(1+\norm{D_h\bfu_h^{n+1}}_h^2+\norm{D_h\bfu_h^n}_h^2\right) \left(1+\sum_{j=1}^{n}\left(\tfrac13\right)^{n-j}
			\norm{\mu_h^{\dagger,j}}_h^2\right).
		\end{aligned}
	\end{equation}

	Taking the cell average of
	\[
		\mu_h^{\dagger,n+1}=\eps^2A_h\phi_h^{\dagger,n+1}
		+\tfrac32G_h^{n+1}-\tfrac12G_h^n
	\]
	and using \((A_h\phi_h^{\dagger,n+1},1)_h=0\), Lemma~\ref{lem:chi-estimates},
	\eqref{eq:mass-conservation}, and Lemma~\ref{lem:base-consequences}, we obtain
	\[
		\begin{aligned}
			\left|\overline{\mu_h^{\dagger,n+1}}\right|
			 & \le C\Bigl(\norm{\chi(\phi_h^{\dagger,n+1},\phi_h^{\dagger,n})}_h
			+\norm{\chi(\phi_h^{\dagger,n},\phi_h^{\dagger,n-1})}_h              \\
			 & \qquad\qquad
			+\norm{\phi_h^{*,n+1}}_h+\norm{\phi_h^{*,n}}_h\Bigr)\le C.
		\end{aligned}
	\]
	Thus, by Lemma~\ref{lem:mac-discrete-estimates} and the triangle inequality,
	\begin{equation}\label{eq:mudagger-l6-bound}
		\norm{\mu_h^{\dagger,n+1}}_{\ell^6}
		\le \norm{\mu_h^{\dagger,n+1}-\overline{\mu_h^{\dagger,n+1}}}_{\ell^6}
		+C\left|\overline{\mu_h^{\dagger,n+1}}\right|
		\le C\left(1+\norm{D_h\mu_h^{\dagger,n+1}}_h\right).
	\end{equation}

	For \(k\ge2\), Lemma~\ref{lem:chi-estimates}, Lemma~\ref{lem:mac-discrete-estimates}
	with \(p=6\), \eqref{eq:mass-conservation}, and Lemma~\ref{lem:base-consequences} give
	\begin{equation}\label{eq:cubic-increment-bound}
		\begin{aligned}
			 & \left|\left(\chi(\phi_h^{\dagger,k+1},\phi_h^{\dagger,k})
			-\chi(\phi_h^{\dagger,k},\phi_h^{\dagger,k-1}),
			\mu_h^{\dagger,n+1}\right)_h\right|                             \\
			 & \quad\le C\norm{\phi_h^{\dagger,k+1}-\phi_h^{\dagger,k-1}}_h
			\norm{\mu_h^{\dagger,n+1}}_{\ell^6}                             \\
			 & \quad\le C\dt\left(\norm{D_{2\tau}\phi_h^{k+1}}_h
			+\norm{D_{2\tau}\phi_h^k}_h\right)
			\norm{\mu_h^{\dagger,n+1}}_{\ell^6},
		\end{aligned}
	\end{equation}
	where \(\phi_h^{\dagger,k+1}-\phi_h^{\dagger,k-1}
	=\dt(D_{2\tau}\phi_h^{k+1}+D_{2\tau}\phi_h^k)\). For the explicit concave part,
	\(\phi_h^{*,k+1}-\phi_h^{*,k}=2(\phi_h^k-\phi_h^{k-1})-(\phi_h^{k-1}-\phi_h^{k-2})\). Solving
	\[
		\phi_h^m-\phi_h^{m-1}
		=\tfrac{2\dt}{3}D_{2\tau}\phi_h^m+\tfrac13(\phi_h^{m-1}-\phi_h^{m-2})
	\]
	gives, for \(m\ge2\),
	\[
		\phi_h^m-\phi_h^{m-1}
		=\tfrac{2\dt}{3}\sum_{j=2}^{m}\left(\tfrac13\right)^{m-j}D_{2\tau}\phi_h^j
		+\left(\tfrac13\right)^{m-1}(\phi_h^1-\phi_h^0).
	\]
	Assumption~\ref{ass:start-up} then implies
	\begin{equation}\label{eq:explicit-concave-increment}
		\norm{\phi_h^{*,k+1}-\phi_h^{*,k}}_h
		\le C\dt\left(1+
		\sum_{j=1}^{k+1}\left(\tfrac13\right)^{k+1-j}
		\norm{D_{2\tau}\phi_h^j}_h\right).
	\end{equation}

	We now estimate the whole \(G_h\)-increment term in \eqref{eq:mu-shifted-start}.
	For \(n\ge3\), using \(G_h^k=\chi(\phi_h^{\dagger,k},\phi_h^{\dagger,k-1})-
	\phi_h^{*,k}\), \eqref{eq:cubic-increment-bound},
	\eqref{eq:explicit-concave-increment}, and \eqref{eq:mudagger-l6-bound}, we have
	\begin{align}
		 & \left|\left(\tfrac32(G_h^{n+1}-G_h^n)-\tfrac12(G_h^n-G_h^{n-1}),
		\mu_h^{\dagger,n+1}\right)_h\right| \notag                             \\
		 & \quad\le C\left|\left(\chi(\phi_h^{\dagger,n+1},\phi_h^{\dagger,n})
		-\chi(\phi_h^{\dagger,n},\phi_h^{\dagger,n-1}),
		\mu_h^{\dagger,n+1}\right)_h\right| \notag                             \\
		 & \qquad+C\left|\left(\chi(\phi_h^{\dagger,n},\phi_h^{\dagger,n-1})
		-\chi(\phi_h^{\dagger,n-1},\phi_h^{\dagger,n-2}),
		\mu_h^{\dagger,n+1}\right)_h\right| \notag                             \\
		 & \qquad+C\left(\norm{\phi_h^{*,n+1}-\phi_h^{*,n}}_h
		+\norm{\phi_h^{*,n}-\phi_h^{*,n-1}}_h\right)
		\norm{\mu_h^{\dagger,n+1}}_{\ell^6} \notag                             \\
		 & \quad\le C\dt\left(\norm{D_{2\tau}\phi_h^{n+1}}_h+
		\norm{D_{2\tau}\phi_h^n}_h+\norm{D_{2\tau}\phi_h^{n-1}}_h\right)
		\left(1+\norm{D_h\mu_h^{\dagger,n+1}}_h\right) \notag                  \\
		 & \qquad+C\dt\left(1+\norm{D_h\mu_h^{\dagger,n+1}}_h\right)
		\left(1+\sum_{j=1}^{n+1}\left(\tfrac13\right)^{n+1-j}
		\norm{D_{2\tau}\phi_h^j}_h\right) \notag                               \\
		 & \quad\le \tfrac{\eps^2\dt}{8}\left(\norm{D_{2\tau}\phi_h^{n+1}}_h^2
		+\norm{D_{2\tau}\phi_h^n}_h^2+\norm{D_{2\tau}\phi_h^{n-1}}_h^2\right)  \\
		 & \qquad+\gamma\dt\sum_{j=1}^{n+1}\left(\tfrac13\right)^{n+1-j}
		\norm{D_{2\tau}\phi_h^j}_h^2
		+C_\gamma\dt\left(1+\norm{D_h\mu_h^{\dagger,n+1}}_h^2
		+\norm{D_h\mu_h^{\dagger,n}}_h^2\right),
		\label{eq:nonlinear-highnorm-estimate}
	\end{align}
	In the last step of \eqref{eq:nonlinear-highnorm-estimate}, we used Young's
	inequality and Cauchy--Schwarz; the term
	\(C_\gamma\dt\norm{D_h\mu_h^{\dagger,n}}_h^2\) has been added for the
	subsequent summation.

	The preceding estimate applies for \(n\ge3\). We next verify the starting case
	\(n=2\). We first expand the transport contribution. For \(k=2,3\),
	\eqref{eq:mac-product-def}, Lemma~\ref{lem:mac-sbp-identities},
	\(d_h\bfu_h^k=0\), H\"older's inequality, and
	Lemma~\ref{lem:mac-discrete-estimates} give
	\[
		\begin{aligned}
			 & \eps^2\dt\left|\left(\tfrac32d_h(\phi_h^{*,3}\bfu_h^3)
			-\tfrac12d_h(\phi_h^{*,2}\bfu_h^2),D_{2\tau}\phi_h^3\right)_h\right|     \\
			 & \quad\le C\dt\sum_{k=2}^{3}
			\norm{\bfu_h^k}_{\ell^4}\norm{D_h\phi_h^{*,k}}_{\ell^4}
			\norm{D_{2\tau}\phi_h^3}_h                                               \\
			 & \quad\le C\dt\sum_{k=2}^{3}
			\norm{\bfu_h^k}_h^{1/2}\norm{D_h\bfu_h^k}_h^{1/2}
			\norm{D_h\phi_h^{*,k}}_h^{1/2}\norm{A_h\phi_h^{*,k}}_h^{1/2}
			\norm{D_{2\tau}\phi_h^3}_h                                               \\
			 & \quad\le C\dt\left(1+\norm{D_h\bfu_h^3}_h+\norm{D_h\bfu_h^2}_h\right)
			\left(1+\norm{A_h\phi_h^{*,3}}_h+\norm{A_h\phi_h^{*,2}}_h\right)
			\norm{D_{2\tau}\phi_h^3}_h                                               \\
			 & \quad\le \tfrac{\eps^2\dt}{16}\norm{D_{2\tau}\phi_h^3}_h^2
			+C\dt\left(1+\norm{D_h\bfu_h^3}_h^2+\norm{D_h\bfu_h^2}_h^2\right)
			\left(1+\norm{A_h\phi_h^{*,3}}_h^2+\norm{A_h\phi_h^{*,2}}_h^2\right).
		\end{aligned}
	\]
	Here, we have utilized \eqref{eq:Ahphi-by-mu},
	\eqref{eq:first-bdf2-highnorm}, and Assumption~\ref{ass:start-up}.
	For the \(G_h\)-term, we separate the regular increment \(G_h^3-G_h^2\) from
	the start-up increment \(G_h^2-G_h^1\):
	\[
		\begin{aligned}
			 & \left|\left(\tfrac32(G_h^3-G_h^2)-\tfrac12(G_h^2-G_h^1),
			\mu_h^{\dagger,3}\right)_h\right|                                    \\
			 & \quad\le C\left|\left(\chi(\phi_h^{\dagger,3},\phi_h^{\dagger,2})
			-\chi(\phi_h^{\dagger,2},\phi_h^{\dagger,1}),
			\mu_h^{\dagger,3}\right)_h\right|                                    \\
			 & \qquad +C\left|\left(\phi_h^{*,3}-\phi_h^{*,2},
			\mu_h^{\dagger,3}\right)_h\right|
			+C\left(\norm{G_h^2}_h+\norm{G_h^1}_h\right)
			\norm{\mu_h^{\dagger,3}}_h                                           \\
			 & \quad\le C\dt\left(1+\norm{D_h\mu_h^{\dagger,3}}_h\right)
			\left(1+\sum_{j=1}^{3}\left(\tfrac13\right)^{3-j}
			\norm{D_{2\tau}\phi_h^j}_h\right)
			+C_T\norm{\mu_h^{\dagger,3}}_h .
		\end{aligned}
	\]
	Here the last line uses \eqref{eq:cubic-increment-bound} and
	\eqref{eq:explicit-concave-increment} with \(k=2\),
	\eqref{eq:mudagger-l6-bound}, and the first-step bound
	\(\norm{G_h^2}_h+\norm{G_h^1}_h\le C_T\). Combining the two estimates and
	applying Young's inequality gives
	\[
		\begin{aligned}
			 & \eps^2\dt\left|\left(\tfrac32d_h(\phi_h^{*,3}\bfu_h^3)
			-\tfrac12d_h(\phi_h^{*,2}\bfu_h^2),D_{2\tau}\phi_h^3\right)_h\right| \\
			 & \quad+\left|\left(\tfrac32(G_h^3-G_h^2)-\tfrac12(G_h^2-G_h^1),
			\mu_h^{\dagger,3}\right)_h\right|                                    \\
			 & \quad\le \tfrac{\eps^2\dt}{8}\norm{D_{2\tau}\phi_h^3}_h^2
			+C\dt\left(1+\norm{D_h\bfu_h^3}_h^2+\norm{D_h\bfu_h^2}_h^2\right)
			\left(1+\norm{A_h\phi_h^{*,3}}_h^2+\norm{A_h\phi_h^{*,2}}_h^2\right) \\
			 & \qquad +\gamma\dt\sum_{j=1}^{3}\left(\tfrac13\right)^{3-j}
			\norm{D_{2\tau}\phi_h^j}_h^2
			+C_\gamma\dt\left(1+\norm{D_h\mu_h^{\dagger,3}}_h^2\right)
			+C_T\norm{\mu_h^{\dagger,3}}_h .
		\end{aligned}
	\]
	Choosing \(\gamma>0\) sufficiently small and using
	\eqref{eq:first-bdf2-highnorm}, Assumption~\ref{ass:start-up},
	Lemma~\ref{lem:base-consequences}, and
	\[
		D_h\mu_h^{\dagger,3}=\tfrac32D_h\mu_h^3-\tfrac12D_h\mu_h^2,
		\quad \dt\norm{D_h\mu_h^{\dagger,3}}_h^2\le C_T,
	\]
	Since \(\phi_h^{*,3}=2\phi_h^2-\phi_h^1\), the linearity of \(A_h\),
	\eqref{eq:Ahphi-by-mu}, \eqref{eq:first-bdf2-highnorm}, and
	Assumption~\ref{ass:start-up} give
	\[
		\norm{A_h\phi_h^{*,3}}_h
		\le
		2\norm{A_h\phi_h^2}_h+\norm{A_h\phi_h^1}_h
		\le
		C\bigl(1+\norm{\mu_h^2}_h\bigr)+C
		\le C_T .
	\]
	Consequently, we obtain
	\[
		\begin{aligned}
			 & \eps^2\dt\left|\left(\tfrac32d_h(\phi_h^{*,3}\bfu_h^3)
			-\tfrac12d_h(\phi_h^{*,2}\bfu_h^2),D_{2\tau}\phi_h^3\right)_h\right| \\
			 & \quad+\left|\left(\tfrac32(G_h^3-G_h^2)-\tfrac12(G_h^2-G_h^1),
			\mu_h^{\dagger,3}\right)_h\right|
			\le \tfrac{\eps^2\dt}{4}\norm{D_{2\tau}\phi_h^3}_h^2
			+\tfrac18\norm{\mu_h^{\dagger,3}}_h^2+C_T .
		\end{aligned}
	\]
	Substituting this estimate into \eqref{eq:mu-shifted-start} with \(n=2\) gives
	\begin{equation}\label{eq:first-shifted-highnorm}
		\norm{\mu_h^{\dagger,3}}_h^2
		+\dt\norm{D_{2\tau}\phi_h^3}_h^2
		\le C_T .
	\end{equation}
	We now sum \eqref{eq:mu-shifted-start}. For \(4\le \ell\le N\), insert
	\eqref{eq:transport-highnorm-estimate} and
	\eqref{eq:nonlinear-highnorm-estimate} and sum over
	\(3\le n\le \ell-1\). The term
	\((\eps^2\dt/4)(\norm{D_{2\tau}\phi_h^{n+1}}_h^2-
	\norm{D_{2\tau}\phi_h^n}_h^2)\) telescopes, and the starting levels are
	controlled by \eqref{eq:first-bdf2-highnorm} and
	\eqref{eq:first-shifted-highnorm}. Also,
	\[
		\begin{aligned}
			 & \dt\sum_{n=3}^{\ell-1}\sum_{j=1}^{n+1}
			\left(\tfrac13\right)^{n+1-j}\norm{D_{2\tau}\phi_h^j}_h^2
			+\dt\sum_{n=3}^{\ell-1}\Bigl(
			\norm{D_{2\tau}\phi_h^{n+1}}_h^2
			+\norm{D_{2\tau}\phi_h^n}_h^2
			+\norm{D_{2\tau}\phi_h^{n-1}}_h^2\Bigr)                          \\
			 & \qquad\le C\dt\sum_{j=1}^{\ell}\norm{D_{2\tau}\phi_h^j}_h^2 .
		\end{aligned}
	\]
	Choosing \(\gamma>0\) sufficiently small, these contributions are absorbed
	into the left-hand side. The remaining terms containing
	\(D_h\mu_h^{\dagger,n}\) are absorbed into \(C_T\), since
	\[
		D_h\mu_h^{\dagger,n}=\tfrac32D_h\mu_h^n-\tfrac12D_h\mu_h^{n-1},
		\quad
		\dt\sum_{n=2}^{N}\norm{D_h\mu_h^{\dagger,n}}_h^2\le C_T
	\]
	by Lemma~\ref{lem:base-consequences}. Hence, after enlarging \(C_T\) to cover
	\(\ell\le3\),
	\begin{equation}\label{eq:weighted-phase-recurrence}
		\begin{aligned}
			 & \norm{\mu_h^{\dagger,\ell}}_h^2
			+\dt\sum_{n=1}^{\ell}\norm{D_{2\tau}\phi_h^n}_h^2 \\
			 & \quad\le C_T+C\dt\sum_{n=1}^{\ell-1}
			\left(1+\norm{D_h\bfu_h^n}_h^2+\norm{D_h\bfu_h^{n+1}}_h^2\right)
			\sum_{j=1}^{n}\left(\tfrac13\right)^{n-j}
			\norm{\mu_h^{\dagger,j}}_h^2 .
		\end{aligned}
	\end{equation}
	The coefficient multiplying the history has bounded time sum by
	Lemma~\ref{lem:base-consequences}. Lemma~\ref{lem:weighted-gronwall} with
	\(\alpha=1/3\) gives
	\[
		\max_{1\le n\le N}\norm{\mu_h^{\dagger,n}}_h^2
		+\dt\sum_{n=1}^{N}\norm{D_{2\tau}\phi_h^n}_h^2\le C_T .
	\]
	Applying \eqref{eq:shifted-bdf2-history} with \(q=\mu_h\) and using the
	start-up bound yields
	\begin{equation}\label{eq:mu-and-dtphi-bound}
		\max_{1\le n\le N}\norm{\mu_h^n}_h^2
		+\dt\sum_{n=1}^{N}\norm{D_{2\tau}\phi_h^n}_h^2\le C_T .
	\end{equation}
	Then \eqref{eq:Ahphi-by-mu} gives
	\(\max_n\norm{A_h\phi_h^n}_h^2\le C_T\), with the levels \(0,1\) covered by
	Assumption~\ref{ass:start-up}. Combining this with \eqref{eq:mass-conservation}
	and Lemma~\ref{lem:mac-discrete-estimates} gives
	\(\max_n\norm{\phi_h^n}_{\ell^\infty}^2\le C_T\).

	It remains to bound \(D_{2\tau}\phi_h^n\) in \(H_h^{-1}\). From the phase
	equation, Lemma~\ref{lem:mac-sbp-identities}, and the definition of
	\(H_h^{-1}\),
	\begin{equation}\label{eq:dtphi-hminus-one}
		\norm{D_{2\tau}\phi_h^{n+1}}_{H_h^{-1}}
		\le \norm{D_h\mu_h^{n+1}}_h
		+C\norm{\phi_h^{*,n+1}}_{\ell^4}\norm{\bfu_h^{n+1}}_{\ell^4} .
	\end{equation}
	The pointwise history formula preceding \eqref{eq:shifted-bdf2-history}, used
	in \(\ell^4\), and Lemma~\ref{lem:base-consequences} imply
	\(\max_n\norm{\phi_h^{*,n}}_{\ell^4}\le C\). Moreover,
	\[
		\norm{\bfu_h^{n+1}}_{\ell^4}^2
		\le C\norm{\bfu_h^{n+1}}_h\norm{D_h\bfu_h^{n+1}}_h
		\le C\left(1+\norm{D_h\bfu_h^{n+1}}_h^2\right).
	\]
	Squaring \eqref{eq:dtphi-hminus-one}, multiplying by \(\dt\), and summing in
	time, Lemma~\ref{lem:base-consequences} gives the bound for the levels
	\(2,\ldots,N\); the level \(n=1\) follows from Assumption~\ref{ass:start-up}.
	Therefore
	\[
		\dt\sum_{n=1}^{N}\norm{D_{2\tau}\phi_h^n}_{H_h^{-1}}^2\le C_T .
	\]
	This proves Lemma~\ref{lem:linfty-phase}.
\end{proof}

\subsection{Convergence}

We now state the fully discrete convergence result. Let
\begin{equation}\label{eq:error-defs}
	\begin{aligned}
		e_\phi^n & =I_c\phi(t_n)-\phi_h^n,
		         & e_\mu^n                   & =R_c\mu(t_n)-\mu_h^n, \\
		e_u^n    & =\Pi_u\bfu(t_n)-\bfu_h^n,
		         & e_p^n                     & =R_p p(t_n)-p_h^n.
	\end{aligned}
\end{equation}
We repeatedly use the approximation estimates of
Lemma~\ref{lem:mac-projections}.

\begin{lemma}\label{lem:consistency}
	Under Assumption~\ref{ass:smooth-solution}, let
	\(\Phi^n=I_c\phi(t_n)\), \(\mathcal M^n=R_c\mu(t_n)\), and
	\(U^n=\Pi_u\bfu(t_n)\). Define the residuals by
	\begin{align*}
		D_{2\tau} U^{n+1}+\mathcal B_h(U^{*,n+1},U^{n+1})
		+\nu A_hU^{n+1} +D_h(R_p p(t_{n+1})) +\Phi^{*,n+1}D_h\mathcal M^{n+1} & =\calR_u^{n+1},    \\
		D_{2\tau}\Phi^{n+1}+A_h\mathcal M^{n+1}
		+d_h(\Phi^{*,n+1}U^{n+1})                                             & =\calR_\phi^{n+1}, \\
		\mathcal M^{n+1}-\eps^2A_h\Phi^{n+1}
		-\chi(\Phi^{\dagger,n+1},\Phi^{\dagger,n})
		+\Phi^{*,n+1}                                                         & =\calR_\mu^{n+1}.
	\end{align*}
	For \(1\le n\le N-1\),
	\begin{equation*}
		\norm{\calR_u^{n+1}}_{H_h^{-1}}
		+\norm{\calR_\phi^{n+1}}_{H_h^{-1}}
		+\norm{\calR_\mu^{n+1}}_h
		+\norm{D_h\calR_\mu^{n+1}}_h
		\le C(\dt^2+h^2).
	\end{equation*}
	Hence
	\begin{equation*}
		\dt\sum_{n=1}^{N-1}\left(
		\norm{\calR_\phi^{n+1}}_{H_h^{-1}}^2+
		\norm{\calR_\mu^{n+1}}_h^2+
		\norm{D_h\calR_\mu^{n+1}}_h^2+
		\norm{\calR_u^{n+1}}_{H_h^{-1}}^2\right)
		\le C(\dt^4+h^4).
	\end{equation*}
\end{lemma}

\begin{proof}
	For any smooth scalar or vector component \(q\), Taylor expansion at
	\(t_{n+1}\) gives
	Assumption~\ref{ass:smooth-solution},
	\begin{equation}\label{eq:time-local-defects}
		\norm{D_{2\tau}q(t_{n+1})-q_t(t_{n+1})}
		+\norm{q^{*,n+1}-q(t_{n+1})}
		\le C\dt^2 .
	\end{equation}
	in every norm required by Assumption~\ref{ass:smooth-solution}. Moreover,
	\begin{align*}
		q^{\dagger,n+1}
		          & =q(t_{n+1})+\tfrac{\dt}{2}q_t(t_{n+1})+\rho_+^{n+1}, \\
		q^{\dagger,n}
		          & =q(t_{n+1})-\tfrac{\dt}{2}q_t(t_{n+1})+\rho_-^{n+1}, \\
		q^{*,n+1} & =q(t_{n+1})+\rho_*^{n+1},
	\end{align*}
	where the remainders are \(O(\dt^2)\). Since \(\chi\) is symmetric, the $O(\dt)$ terms cancel and
	\begin{equation}\label{eq:shifted-nonlinear-consistency}
		\norm{\chi(q^{\dagger,n+1},q^{\dagger,n})-q^{*,n+1}
			-f(q(t_{n+1}))}
		\le C\dt^2,
	\end{equation}
	The same estimate holds after one spatial derivative.

	The scalar consistency estimates are given in
	Lemma~\ref{lem:mac-projections}. For the pressure gradient,
	\begin{equation}\label{eq:weak-pressure-consistency}
		\norm{D_h R_p q-R_e\nabla q}_{H_h^{-1}}
		\le Ch^2\norm{q}_{W^{2,\infty}},
		\qquad (q,1)_{L^2}=0 .
	\end{equation}
	For the velocity projection, write \(\Pi_u\bfv-R_e\bfv=D_h\theta_{\bfv}\) and
	\(\delta_h=d_hR_e\bfv-R_c(\nabla\cdot\bfv)\). Since \(A_h\theta_{\bfv}=\delta_h\) and the discrete difference operators
	commute,
	\[
		A_h(\Pi_u\bfv-R_e\bfv)
		=
		D_h\delta_h.
	\]
	Therefore,
	\begin{equation}\label{eq:weak-velocity-projection-consistency}
		\norm{A_h(\Pi_u\bfv-R_e\bfv)}_{H_h^{-1}}
		=\norm{D_h\delta_h}_{H_h^{-1}}
		\le \norm{\delta_h}_h
		\le Ch^2\norm{\bfv}_{W^{3,\infty}} .
	\end{equation}
	Moreover,
	\[
		\norm{D_h(\Pi_u\bfv-R_e\bfv)}_{\ell^\infty}
		\le Ch^{-1}\norm{D_h(\Pi_u\bfv-R_e\bfv)}_h
		\le C\norm{\bfv}_{W^{3,\infty}}.
	\]

	For the phase equation, \eqref{eq:time-local-defects},
	Lemma~\ref{lem:mac-projections}, and the product estimate give
	\begin{align*}
		 & \norm{D_{2\tau}\Phi^{n+1}-R_c\phi_t(t_{n+1})}_{H_h^{-1}}
		+\norm{A_h\mathcal M^{n+1}-R_c(-\Delta\mu(t_{n+1}))}_{H_h^{-1}} \\
		 & \quad
		+\norm{d_h(\Phi^{*,n+1}U^{n+1})
			 -R_c\nabla\cdot(\phi\bfu)(t_{n+1})}_{H_h^{-1}}
		\le C(\dt^2+h^2).
	\end{align*}
	Hence
	\(\norm{\calR_\phi^{n+1}}_{H_h^{-1}}\le C(\dt^2+h^2)\).

	For the momentum equation, \eqref{eq:time-local-defects}, the MAC convection
	consistency estimate, \eqref{eq:weak-pressure-consistency}, and
	\eqref{eq:weak-velocity-projection-consistency} give
	\begin{align*}
		 & \norm{D_{2\tau}U^{n+1}-R_e\bfu_t(t_{n+1})}_{H_h^{-1}}
		+\nu\norm{A_hU^{n+1}-R_e(-\Delta\bfu)(t_{n+1})}_{H_h^{-1}} \\
		 & \quad
		+\norm{\mathcal B_h(U^{*,n+1},U^{n+1})
			 -R_e((\bfu\cdot\nabla)\bfu)(t_{n+1})}_{H_h^{-1}}         \\
		 & \quad
		+\norm{\Phi^{*,n+1}D_h\mathcal M^{n+1}
			 -R_e(\phi\nabla\mu)(t_{n+1})}_{H_h^{-1}}
		+\norm{D_h R_p p(t_{n+1})-R_e\nabla p(t_{n+1})}_{H_h^{-1}} \\
		 & \le C(\dt^2+h^2).
	\end{align*}
	The continuous momentum equation gives
	\(\norm{\calR_u^{n+1}}_{H_h^{-1}}\le C(\dt^2+h^2)\).

	For the chemical potential, use
	\(\mu=-\eps^2\Delta\phi+f(\phi)\). Since \(I_c\phi-R_c\phi\) is a spatially
	constant of size \(O(h^2)\), the elliptic part is controlled by
	Lemma~\ref{lem:mac-projections}. The nonlinear defect is decomposed as
	\begin{align*}
		 & \chi(\Phi^{\dagger,n+1},\Phi^{\dagger,n})-
		\Phi^{*,n+1}-R_cf(\phi(t_{n+1}))                          \\
		 & \quad=\Bigl[\chi(\Phi^{\dagger,n+1},\Phi^{\dagger,n})-
			         \Phi^{*,n+1}
			         -R_c\{\chi(\phi^{\dagger,n+1},\phi^{\dagger,n})-
			         \phi^{*,n+1}\}\Bigr] \\
		 & \qquad
		+R_c\{\chi(\phi^{\dagger,n+1},\phi^{\dagger,n})-
		\phi^{*,n+1}-f(\phi(t_{n+1}))\} .
	\end{align*}
	The first bracket is bounded by \(Ch^2\) in both
	\(\ell^2\) and the discrete gradient norm, while the second is bounded by
	\(C\dt^2\) by \eqref{eq:shifted-nonlinear-consistency}. Hence
	\begin{equation*}
		\norm{\calR_\mu^{n+1}}_h+
		\norm{D_h\calR_\mu^{n+1}}_h
		\le C(\dt^2+h^2).
	\end{equation*}
	Multiplying by \(\dt\), summing over \(n\), and using \(N\dt=T\) gives the
	summed residual bound in Lemma~\ref{lem:consistency}
\end{proof}

Let \(\Phi^n=I_c\phi(t_n)\), \(\mathcal M^n=R_c\mu(t_n)\), and
\(U^n=\Pi_u\bfu(t_n)\). Subtracting
\eqref{eq:scheme} from the consistency equations in Lemma~\ref{lem:consistency} gives
\begin{subequations}\label{eq:error-system}
	\begin{align}
		D_{2\tau} e_u^{n+1}+\nu A_he_u^{n+1}+D_he_p^{n+1}
		             & =\calN_u^{n+1}+\calC_u^{n+1}+\calR_u^{n+1}, \label{eq:error-u} \\
		d_he_u^{n+1} & =0,\label{eq:error-div}                                        \\
		D_{2\tau} e_\phi^{n+1}+A_he_\mu^{n+1}
		             & =\calC_\phi^{n+1}+\calR_\phi^{n+1}, \label{eq:error-phi}       \\
		e_\mu^{n+1}-\eps^2A_he_\phi^{n+1}
		-\left(\tfrac12(e_\phi^{\dagger,n+1}+e_\phi^{\dagger,n})
		-e_\phi^{*,n+1}\right)
		             & =\calN_\phi^{n+1}+\calR_\mu^{n+1}. \label{eq:error-mu}
	\end{align}
\end{subequations}
Here
\begin{align}
	\calN_u^{n+1}
	 & =-\mathcal B_h(U^{*,n+1},U^{n+1})
	+\mathcal B_h(\bfu_h^{*,n+1},\bfu_h^{n+1}),\label{eq:def-Nu}           \\
	\calN_\phi^{n+1}
	 & =\chi(\Phi^{\dagger,n+1},\Phi^{\dagger,n})
	-\chi(\phi_h^{\dagger,n+1},\phi_h^{\dagger,n})
	-\tfrac12(e_\phi^{\dagger,n+1}+e_\phi^{\dagger,n}),\label{eq:def-Nphi} \\
	\calC_u^{n+1}
	 & =-\Phi^{*,n+1}D_h\mathcal M^{n+1}
	+\phi_h^{*,n+1}D_h\mu_h^{n+1},\label{eq:def-Cu}                        \\
	\calC_\phi^{n+1}
	 & =-d_h\left(\Phi^{*,n+1}U^{n+1}
	-\phi_h^{*,n+1}\bfu_h^{n+1}\right).\label{eq:def-Cphi}
\end{align}

\begin{lemma}\label{lem:nonlinear-estimates}
	Under Assumptions~\ref{ass:start-up} and~\ref{ass:smooth-solution}, there is a
	constant \(C_T\), independent of \(h\) and \(\dt\), such that, for
	\(n\ge1\),
	\begin{equation*}
		\begin{aligned}
			\left|(\calN_u^{n+1},e_u^{n+1})_h\right|
			 & \le \tfrac{\nu}{8}\norm{D_he_u^{n+1}}_h^2 +C_T\left(\norm{e_u^{n+1}}_h^2+\norm{e_u^n}_h^2
			+\norm{e_u^{n-1}}_h^2\right)+C_T(\dt^4+h^4),
		\end{aligned}
	\end{equation*}
	\begin{equation*}
		\begin{aligned}
			\left|(\calN_\phi^{n+1},D_{2\tau} e_\phi^{n+1})_h\right|
			 & \le \tfrac18\norm{D_he_\mu^{n+1}}_h^2  +C_T\left(\norm{D_he_\phi^{n+1}}_h^2
			+\norm{D_he_\phi^n}_h^2+\norm{D_he_\phi^{n-1}}_h^2
			+\norm{e_u^{n+1}}_h^2\right)                                                   \\
			 & \quad+C_T(\dt^4+h^4),
		\end{aligned}
	\end{equation*}
	\begin{equation*}
		\begin{aligned}
			 & \left|(\calC_u^{n+1},e_u^{n+1})_h
			+(\calC_\phi^{n+1},e_\mu^{n+1})_h\right| \le \tfrac{\nu}{8}\norm{D_he_u^{n+1}}_h^2
			+\tfrac18\norm{D_he_\mu^{n+1}}_h^2                       \\
			 & \qquad+C_T\left(\norm{e_u^{n+1}}_h^2+\norm{e_u^n}_h^2
			+\norm{e_u^{n-1}}_h^2\right) +C_T\left(\norm{D_he_\phi^n}_h^2 +\norm{D_he_\phi^{n-1}}_h^2\right)+C_T(\dt^4+h^4).
		\end{aligned}
	\end{equation*}
\end{lemma}

\begin{proof}
	The phase error satisfies
	\[
		(e_\phi^j,1)_h=0,
	\]
	by mass conservation. Moreover,
	\[
		\norm{U^j}_{\ell^\infty}
		+\norm{D_hU^j}_{\ell^\infty}
		\le C_T,
	\]
	which follows from the projection estimate
	\eqref{eq:weak-velocity-projection-consistency} and
	Assumption~\ref{ass:smooth-solution}.

	\emph{Convective defect.}  Since
	\(d_h\bfu_h^{*,n+1}=d_hU^{*,n+1}=0\), Lemma~\ref{lem:mac-convection-form} gives
	\begin{align*}
		(\calN_u^{n+1},e_u^{n+1})_h
		 & =b_h(\bfu_h^{*,n+1},U^{n+1},e_u^{n+1})
		-b_h(U^{*,n+1},U^{n+1},e_u^{n+1})         \\
		 & =-b_h(e_u^{*,n+1},U^{n+1},e_u^{n+1}).
	\end{align*}
	Hence,
	\begin{align*}
		|b_h(e_u^{*,n+1},U^{n+1},e_u^{n+1})|
		 & \le C_T\norm{e_u^{*,n+1}}_h\norm{e_u^{n+1}}_h \\
		 & \le C_T\left(\norm{e_u^{n+1}}_h^2+
		\norm{e_u^n}_h^2
		+\norm{e_u^{n-1}}_h^2\right),
	\end{align*}
	where we used the boundedness of \(D_hU^{n+1}\).

	\emph{Double-well defect.}  From \eqref{eq:def-Nphi},
	\begin{equation}\label{eq:Nphi-gradient-bound}
		\begin{aligned}
			\norm{D_h\calN_\phi^{n+1}}_h
			 & \le \norm{D_h\left(\chi(\Phi^{\dagger,n+1},\Phi^{\dagger,n})
				       -\chi(\phi_h^{\dagger,n+1},\phi_h^{\dagger,n})\right)}_h        \\
			 & \quad +\tfrac12\norm{D_h(e_\phi^{\dagger,n+1}+e_\phi^{\dagger,n})}_h \\
			 & \le C_T\left(\norm{D_he_\phi^{\dagger,n+1}}_h+
			\norm{D_he_\phi^{\dagger,n}}_h\right)                                   \\
			 & \le C_T\left(\norm{D_he_\phi^{n+1}}_h+
			\norm{D_he_\phi^n}_h
			+\norm{D_he_\phi^{n-1}}_h\right).
		\end{aligned}
	\end{equation}
	Here we used
	Lemmas~\ref{lem:linfty-phase},
	\ref{lem:mac-discrete-estimates},
	and~\ref{lem:chi-estimates}, together with
	Assumption~\ref{ass:smooth-solution}.

	For \(\calC_\phi^{n+1}\), \eqref{eq:def-Cphi} yields
	\begin{equation}\label{eq:Cphi-hminus-one}
		\begin{aligned}
			\calC_\phi^{n+1}
			 & =-d_h\left(\Phi^{*,n+1}e_u^{n+1}
			+e_\phi^{*,n+1}U^{n+1}-e_\phi^{*,n+1}e_u^{n+1}\right),      \\
			\norm{\calC_\phi^{n+1}}_{H_h^{-1}}
			 & \le C\norm{\Phi^{*,n+1}e_u^{n+1}
				        +e_\phi^{*,n+1}U^{n+1}-e_\phi^{*,n+1}e_u^{n+1}}_h, \\
			\norm{\calC_\phi^{n+1}}_{H_h^{-1}}^2
			 & \le C_T\left(\norm{e_u^{n+1}}_h^2+
			\norm{D_he_\phi^n}_h^2
			+\norm{D_he_\phi^{n-1}}_h^2\right).
		\end{aligned}
	\end{equation}
	By \eqref{eq:error-phi},
	\begin{equation*}
		\norm{D_{2\tau}e_\phi^{n+1}}_{H_h^{-1}}
		\le \norm{D_he_\mu^{n+1}}_h
		+\norm{\calC_\phi^{n+1}}_{H_h^{-1}}
		+\norm{\calR_\phi^{n+1}}_{H_h^{-1}} .
	\end{equation*}
	Since \(D_{2\tau}e_\phi^{n+1}\) has zero mean,
	\begin{align*}
		|(\calN_\phi^{n+1},D_{2\tau}e_\phi^{n+1})_h|
		 & \le \norm{D_h\calN_\phi^{n+1}}_h
		\norm{D_{2\tau}e_\phi^{n+1}}_{H_h^{-1}} .
	\end{align*}
	Combining
	\eqref{eq:Nphi-gradient-bound},
	\eqref{eq:Cphi-hminus-one},
	Lemma~\ref{lem:consistency},
	and Young's inequality yields the second estimate in
	Lemma~\ref{lem:nonlinear-estimates}.

	\emph{Coupling defect.}  Using
	\(\phi_h^{*,n+1}=\Phi^{*,n+1}-e_\phi^{*,n+1}\),
	\(\bfu_h^{n+1}=U^{n+1}-e_u^{n+1}\), and
	\(\mu_h^{n+1}=\mathcal M^{n+1}-e_\mu^{n+1}\), we obtain
	\begin{align*}
		\calC_u^{n+1}
		 & =-\Phi^{*,n+1}D_he_\mu^{n+1}
		-e_\phi^{*,n+1}D_h\mathcal M^{n+1}
		+e_\phi^{*,n+1}D_he_\mu^{n+1},      \\
		\calC_\phi^{n+1}
		 & =-d_h\left(\Phi^{*,n+1}e_u^{n+1}
		+e_\phi^{*,n+1}U^{n+1}-e_\phi^{*,n+1}e_u^{n+1}\right).
	\end{align*}
	Moreover,
	\begin{equation*}
		- (\Phi^{*,n+1}D_he_\mu^{n+1},e_u^{n+1})_h
		-(d_h(\Phi^{*,n+1}e_u^{n+1}),e_\mu^{n+1})_h=0,
	\end{equation*}
	by Lemma~\ref{lem:mac-sbp-identities}. Therefore
	\begin{align*}
		|(e_\phi^{*,n+1}D_h\mathcal M^{n+1},e_u^{n+1})_h|
		 & \le C_T\left(\norm{D_he_\phi^n}_h^2+
		\norm{D_he_\phi^{n-1}}_h^2
		+\norm{e_u^{n+1}}_h^2\right),               \\
		|(e_\phi^{*,n+1}D_he_\mu^{n+1},e_u^{n+1})_h|
		 & \le \tfrac1{32}\norm{D_he_\mu^{n+1}}_h^2
		+\tfrac{\nu}{32}\norm{D_he_u^{n+1}}_h^2     \\
		 & \quad +C_T\left(\norm{e_u^{n+1}}_h^2+
		\norm{D_he_\phi^n}_h^2
		+\norm{D_he_\phi^{n-1}}_h^2\right),         \\
		|(e_\phi^{*,n+1}U^{n+1},D_he_\mu^{n+1})_h|
		 & \le \tfrac1{32}\norm{D_he_\mu^{n+1}}_h^2
		+C_T\left(\norm{D_he_\phi^n}_h^2+
		\norm{D_he_\phi^{n-1}}_h^2\right),          \\
		|(e_\phi^{*,n+1}e_u^{n+1},D_he_\mu^{n+1})_h|
		 & \le \tfrac1{32}\norm{D_he_\mu^{n+1}}_h^2
		+C_T\norm{e_u^{n+1}}_h^2 .
	\end{align*}
	Using Poincar\'e's inequality, Lemma~\ref{lem:mac-discrete-estimates}, and
	the uniform bounds for \(D_h\mathcal M^{n+1}\), \(U^{n+1}\), and
	\(e_\phi^{*,n+1}\), these estimates yield
	the third bound in Lemma~\ref{lem:nonlinear-estimates}.
\end{proof}

Define the error energy by
\begin{align}\label{eq:error-energy}
	\mathfrak E^{n}
	=\calG(e_u^n,e_u^{n-1})
	+\eps^2\calG(D_he_\phi^n,D_he_\phi^{n-1})
	+\tfrac38\norm{e_\phi^n-e_\phi^{n-1}}_h^2.
\end{align}

\begin{lemma}[Discrete velocity projection]\label{lem:mac-leray}
	Let
	\[
		Z_h=\{\boldsymbol z_h\in\mathsf E_h:d_h\boldsymbol z_h=0\}.
	\]
	For \(\bfv_h\in\mathsf E_h\), define \(P_h\bfv_h\) by
	\begin{equation*}
		P_h\bfv_h=\bfv_h+D_h\pi_h,
		\quad
		A_h\pi_h=d_h\bfv_h,
		\quad
		(\pi_h,1)_h=0 .
	\end{equation*}
	Then \(P_h\bfv_h\in Z_h\) and, for \(q_h\in\mathsf C_h\),
	\begin{equation*}
		P_hD_hq_h=0,
		\quad
		\norm{P_h\bfv_h}_{H_h^{-1}}\le C\norm{\bfv_h}_{H_h^{-1}},
		\quad
		\norm{A_h\bfv_h}_{H_h^{-1}}\le C\norm{D_h\bfv_h}_h .
	\end{equation*}
	Here \(A_h\) acts componentwise on vector fields. This operator is distinct
	from the smooth-velocity projection \(\Pi_u\) used in the approximation
	estimates; it is used only to eliminate the discrete pressure gradient in the
	pressure estimate below. The proof is given in Appendix~\ref{app:mac-leray}.
\end{lemma}

For scalar physical errors in the theorem below, we use
\[
	\norm{\psi-\varphi_h}_{H_h^1}
	:=\norm{R_c\psi-\varphi_h}_h+
	\norm{R_e\nabla\psi-D_h\varphi_h}_h,
	\quad \psi\in C^1(\overline\Omega),\quad \varphi_h\in\mathsf C_h .
\]

\begin{theorem}[Second-order convergence]\label{thm:main-convergence}
	Let Assumptions~\ref{ass:start-up} and~\ref{ass:smooth-solution} hold. Then
	there exist \(h_0>0\), \(\dt_0>0\), and \(C_T>0\), independent of
	\(h\) and \(\dt\), such that, for \(0<h\le h_0\), \(0<\dt\le\dt_0\), and
	\(N\dt=T\), the solution of \eqref{eq:scheme} satisfies
	\begin{align}\label{eq:main-estimate}
		\max_{1\le n\le N}\mathfrak E^n
		+\dt\sum_{n=1}^{N}\left(\norm{D_he_\mu^n}_h^2
		+\nu\norm{D_he_u^n}_h^2\right)
		\le C_T(\dt^4+h^4).
	\end{align}
	Consequently,
	\begin{equation}\label{eq:physical-error-estimate}
		\begin{aligned}
			 & \max_{1\le n\le N}\left(\norm{\phi(t_n)-\phi_h^n}_{H_h^1}
			+\norm{R_e\bfu(t_n)-\bfu_h^n}_h\right) +\left(\dt\sum_{n=1}^N\norm{\mu(t_n)-\mu_h^n}_{H_h^1}^2\right)^{1/2}
			\le C_T(\dt^2+h^2),
		\end{aligned}
	\end{equation}
	and
	\begin{equation}\label{eq:pressure-estimate}
		\left(\dt\sum_{n=1}^N\norm{R_c p(t_n)-p_h^n}_h^2\right)^{1/2}
		\le C_T(\dt^2+h^2).
	\end{equation}
\end{theorem}

\begin{proof}
	Testing \eqref{eq:error-u} by \(\dt e_u^{n+1}\),
	\eqref{eq:error-phi} by \(\dt e_\mu^{n+1}\), and
	\eqref{eq:error-mu} by \(\dt D_{2\tau}e_\phi^{n+1}\), then subtracting the
	last identity from the second and adding the first, yields the basic error
	relation. The pressure term vanishes by
	Lemma~\ref{lem:mac-sbp-identities} and \eqref{eq:error-div}, while the two
	\((e_\mu^{n+1},D_{2\tau}e_\phi^{n+1})_h\) terms cancel. Moreover,
	\[
		\tfrac12(e_\phi^{\dagger,n+1}+e_\phi^{\dagger,n})-e_\phi^{*,n+1}
		=\tfrac34\delta^2e_\phi^{n+1},
	\]
	and therefore
	\[
		\begin{aligned}
			 & \dt\left(\tfrac12(e_\phi^{\dagger,n+1}+e_\phi^{\dagger,n})
			-e_\phi^{*,n+1},D_{2\tau}e_\phi^{n+1}\right)_h                \\
			 & \quad=\tfrac38\left(\norm{e_\phi^{n+1}-e_\phi^n}_h^2
			-\norm{e_\phi^n-e_\phi^{n-1}}_h^2\right)
			+\tfrac34\norm{\delta^2e_\phi^{n+1}}_h^2 .
		\end{aligned}
	\]
	Using Lemma~\ref{lem:bdf2-algebra}, we obtain
	\begin{align}\label{eq:error-energy-recursion}
		 & \mathfrak E^{n+1}-\mathfrak E^n
		+\tfrac14\norm{\delta^2e_u^{n+1}}_h^2
		+\tfrac{\eps^2}{4}\norm{D_h\delta^2e_\phi^{n+1}}_h^2
		+\tfrac34\norm{\delta^2e_\phi^{n+1}}_h^2\notag                    \\
		 & \quad
		+\dt\norm{D_he_\mu^{n+1}}_h^2
		+\nu\dt\norm{D_he_u^{n+1}}_h^2 \notag                             \\
		 & \quad
		\le \dt\left|(\calN_u^{n+1},e_u^{n+1})_h\right|
		+\dt\left|(\calN_\phi^{n+1},D_{2\tau}e_\phi^{n+1})_h\right|\notag \\
		 & \qquad
		+\dt\left|(\calC_u^{n+1},e_u^{n+1})_h
		+(\calC_\phi^{n+1},e_\mu^{n+1})_h\right|\notag                    \\
		 & \qquad
		+\dt |(\calR_u^{n+1},e_u^{n+1})_h|
		+\dt |(\calR_\phi^{n+1},e_\mu^{n+1})_h|\notag                     \\
		 & \qquad
		+\dt |(\calR_\mu^{n+1},D_{2\tau}e_\phi^{n+1})_h| .
	\end{align}
	The first two residual terms satisfy
	\begin{align*}
		\dt |(\calR_u^{n+1},e_u^{n+1})_h|
		 & \le \tfrac{\nu\dt}{32}\norm{D_he_u^{n+1}}_h^2
		+C\dt\norm{\calR_u^{n+1}}_{H_h^{-1}}^2,          \\
		\dt |(\calR_\phi^{n+1},e_\mu^{n+1})_h|
		 & \le \tfrac{\dt}{32}\norm{D_he_\mu^{n+1}}_h^2
		+C\dt\norm{\calR_\phi^{n+1}}_{H_h^{-1}}^2.
	\end{align*}
	For the last residual term, \(D_{2\tau}e_\phi^{n+1}\) has zero mean, so
	\begin{align*}
		\dt |(\calR_\mu^{n+1},D_{2\tau}e_\phi^{n+1})_h|
		 & \le \dt\norm{D_h\calR_\mu^{n+1}}_h
		\norm{D_{2\tau}e_\phi^{n+1}}_{H_h^{-1}}   \\
		 & \le \dt\norm{D_h\calR_\mu^{n+1}}_h
		\Bigl(\norm{D_he_\mu^{n+1}}_h
		+\norm{\calC_\phi^{n+1}}_{H_h^{-1}}\Bigr) \\
		 & \quad +\dt\norm{D_h\calR_\mu^{n+1}}_h
		\norm{\calR_\phi^{n+1}}_{H_h^{-1}}.
	\end{align*}
	Using \eqref{eq:Cphi-hminus-one} and Young's inequality, the residual terms in
	\eqref{eq:error-energy-recursion} are bounded by absorbable parts of
	\(\nu\dt\norm{D_he_u^{n+1}}_h^2\) and
	\(\dt\norm{D_he_\mu^{n+1}}_h^2\), by
	\(C_T\dt(\mathfrak E^{n+1}+\mathfrak E^n+\mathfrak E^{n-1})\), and by
	\begin{equation}\label{eq:residual-energy-bound}
		C\dt\left(\norm{\calR_u^{n+1}}_{H_h^{-1}}^2+
		\norm{\calR_\phi^{n+1}}_{H_h^{-1}}^2+
		\norm{D_h\calR_\mu^{n+1}}_h^2\right).
	\end{equation}
	By Lemma~\ref{lem:consistency}, the time sum of
	\eqref{eq:residual-energy-bound} is \(O(\dt^4+h^4)\).

	The \(G\)-functional satisfies
	\begin{equation}\label{eq:g-equivalence}
		c_G\left(\norm{a}_h^2+\norm{b}_h^2\right)
		\le \calG(a,b)
		\le C_G\left(\norm{a}_h^2+\norm{b}_h^2\right),
	\end{equation}
	with constants independent of \(h\). Consequently,
	\begin{equation*}
		\norm{e_u^n}_h^2+
		\norm{e_u^{n-1}}_h^2
		+\norm{D_he_\phi^n}_h^2+\norm{D_he_\phi^{n-1}}_h^2
		+\norm{e_\phi^n-e_\phi^{n-1}}_h^2 .
	\end{equation*}
	Substituting Lemma~\ref{lem:nonlinear-estimates} into
	\eqref{eq:error-energy-recursion}, using \eqref{eq:g-equivalence}, and
	absorbing the dissipative fractions yield, for \(n\ge2\),
	\begin{equation}\label{eq:gronwall-ready}
		\begin{aligned}
			 & \mathfrak E^{n+1}-\mathfrak E^n
			+c\dt\left(\norm{D_he_\mu^{n+1}}_h^2
			+\nu\norm{D_he_u^{n+1}}_h^2\right) \\
			 & \quad
			\le C_T\dt(\mathfrak E^{n+1}+\mathfrak E^n+\mathfrak E^{n-1})
			+C_T\dt(\dt^4+h^4).
		\end{aligned}
	\end{equation}
	For \(2\le m\le N\), summing \eqref{eq:gronwall-ready} from \(n=2\) to
	\(m-1\) gives
	\begin{align*}
		 & \mathfrak E^m-
		\mathfrak E^2
		+c\dt\sum_{n=2}^{m-1}
		\left(\norm{D_he_\mu^{n+1}}_h^2
		+\nu\norm{D_he_u^{n+1}}_h^2\right) \\
		 & \quad
		\le C_T\dt\sum_{n=2}^{m-1}
		(\mathfrak E^{n+1}+\mathfrak E^n+\mathfrak E^{n-1})
		+C_T(m-2)\dt(\dt^4+h^4).
	\end{align*}
	Choose \(\dt_0 > 0\) so that \(C_T\dt_0\le1/4\). Moving the term
	\(C_T\dt\mathfrak E^m\) to the left and reindexing the remaining
	history terms yields
	\begin{align*}
		 & \mathfrak E^m
		+c\dt\sum_{n=2}^{m-1}
		\left(\norm{D_he_\mu^{n+1}}_h^2
		+\nu\norm{D_he_u^{n+1}}_h^2\right) \\
		 & \quad
		\le C_T(\mathfrak E^1+\mathfrak E^2)
		+C_T\dt\sum_{j=2}^{m-1}\mathfrak E^j
		+C_T(\dt^4+h^4).
	\end{align*}
	The corresponding estimate at \(n=1\), together with
	Assumption~\ref{ass:start-up}, bounds \(\mathfrak E^2\) and the level-two
	dissipation. The start-up bound for \(\mathfrak E^1\), the preceding
	inequality, Lemma~\ref{lem:discrete-gronwall} applied to the reindexed sequence
	\(a_\ell=\mathfrak E^{\ell+1}\), and Lemma~\ref{lem:consistency} prove
	\eqref{eq:main-estimate}.

	The velocity and phase estimates in
	\eqref{eq:physical-error-estimate} follow from
	\eqref{eq:main-estimate},
	Poincar\'e's inequality for the zero-mean phase error,
	and Lemma~\ref{lem:mac-projections}. For the chemical potential,
	averaging \eqref{eq:error-mu} over the grid gives
	\[
		\overline{e_\mu^{n+1}}
		=
		\overline{\calN_\phi^{n+1}}
		+\overline{\calR_\mu^{n+1}},
	\]
	since
	\[
		(A_he_\phi^{n+1},1)_h=0,
		\quad
		\left(
		\tfrac12(e_\phi^{\dagger,n+1}+e_\phi^{\dagger,n})
		-e_\phi^{*,n+1},
		1
		\right)_h=0.
	\]
	The local Lipschitz bound for \(\chi\), Poincar\'e's inequality, and
	Lemma~\ref{lem:consistency} therefore imply
	\begin{equation}\label{eq:emu-mean-recovery}
		|\overline{e_\mu^{n+1}}|
		\le C_T\left(\norm{D_he_\phi^{n+1}}_h+
		\norm{D_he_\phi^n}_h
		+\norm{D_he_\phi^{n-1}}_h+\dt^2+h^2\right).
	\end{equation}
	Together with
	\begin{equation*}
		\norm{e_\mu^{n+1}}_h
		\le C\norm{D_he_\mu^{n+1}}_h+C|\overline{e_\mu^{n+1}}|,
	\end{equation*}
	this yields the chemical-potential estimate in
	\eqref{eq:physical-error-estimate}, after accounting for the restriction
	error in Lemma~\ref{lem:mac-projections}.

	For the pressure, Lemma~\ref{lem:mac-infsup} gives
	\begin{align*}
		\beta\norm{e_p^{n+1}}_h
		 & \le \sup_{\bfv_h\ne0}
		\tfrac{|(e_p^{n+1},d_h\bfv_h)_h|}{\norm{D_h\bfv_h}_h}
		=\sup_{\bfv_h\ne0}
		\tfrac{|(D_he_p^{n+1},\bfv_h)_h|}{\norm{D_h\bfv_h}_h}.
	\end{align*}
	Since
	\[
		P_hD_{2\tau}e_u^{n+1}
		=
		D_{2\tau}e_u^{n+1},
		\quad
		P_hD_he_p^{n+1}
		=
		0,
	\]
	applying \(P_h\) to \eqref{eq:error-u} and using
	Lemma~\ref{lem:mac-leray} yields
	\begin{equation}\label{eq:velocity-time-dual-bound-pointwise}
		\begin{aligned}
			\norm{D_{2\tau}e_u^{n+1}}_{H_h^{-1}}
			\le C\bigl( & \nu\norm{D_he_u^{n+1}}_h
			+\norm{\calN_u^{n+1}}_{H_h^{-1}} +\norm{\calC_u^{n+1}}_{H_h^{-1}}
			+\norm{\calR_u^{n+1}}_{H_h^{-1}}\bigr).
		\end{aligned}
	\end{equation}
	The error momentum equation and \eqref{eq:velocity-time-dual-bound-pointwise}
	bound \(\norm{e_p^{n+1}}_h\) by the same four terms on the right.

	For the convection defect, use
	\begin{equation*}
		\calN_u^{n+1}
		=-\mathcal B_h(U^{*,n+1},e_u^{n+1})
		-\mathcal B_h(e_u^{*,n+1},U^{n+1})
		+\mathcal B_h(e_u^{*,n+1},e_u^{n+1}).
	\end{equation*}
	Lemma~\ref{lem:mac-convection-estimates} gives
	\[
		\norm{\mathcal B_h(U^{*,n+1},e_u^{n+1})}_{H_h^{-1}}
		+\norm{\mathcal B_h(e_u^{*,n+1},U^{n+1})}_{H_h^{-1}}
		\le C_T\left(\norm{e_u^{n+1}}_h+\norm{e_u^{*,n+1}}_h\right).
	\]
	For the last term,
	\begin{equation*}
		\norm{\mathcal B_h(e_u^{*,n+1},e_u^{n+1})}_{H_h^{-1}}^2
		\le C\norm{e_u^{*,n+1}}_h\norm{D_he_u^{*,n+1}}_h
		\norm{e_u^{n+1}}_h\norm{D_he_u^{n+1}}_h .
	\end{equation*}
	Using \eqref{eq:main-estimate}, Cauchy--Schwarz, and the
	corresponding bounds for the extrapolated error, we obtain
	\begin{equation}\label{eq:Nu-hminus-pressure-sum}
		\dt\sum_{n=1}^{N-1}\norm{\calN_u^{n+1}}_{H_h^{-1}}^2
		\le C_T(\dt^4+h^4).
	\end{equation}
	The decomposition of \(\calC_u^{n+1}\) gives
	\begin{equation*}
		\norm{\calC_u^{n+1}}_{H_h^{-1}}^2
		\le C_T\left(\norm{D_he_\mu^{n+1}}_h^2+
		\norm{D_he_\phi^n}_h^2
		+\norm{D_he_\phi^{n-1}}_h^2\right),
	\end{equation*}
	Together with
	\eqref{eq:main-estimate}, \eqref{eq:Nu-hminus-pressure-sum}, and
	Lemma~\ref{lem:consistency}, this gives
	\begin{equation}\label{eq:velocity-time-difference-bound}
		\dt\sum_{n=1}^{N-1}\norm{D_{2\tau}e_u^{n+1}}_{H_h^{-1}}^2
		\le C_T(\dt^4+h^4),
	\end{equation}
	and, after multiplying the inf-sup estimate by \(\dt\) and summing over
	\(n=1,\ldots,N-1\),
	\begin{equation*}
		\dt\sum_{n=1}^{N-1}\norm{e_p^{n+1}}_h^2
		\le C_T(\dt^4+h^4).
	\end{equation*}
	The level \(n=1\) is covered by Assumption~\ref{ass:start-up}. Adding
	\(\norm{R_c p(t_n)-R_p p(t_n)}_h\le Ch^2\) from
	Lemma~\ref{lem:mac-projections} proves \eqref{eq:pressure-estimate}.
\end{proof}

\section{Numerical experiments}

A smooth manufactured solution on \(\Omega=(0,1)^2\) is used to verify the
predicted convergence rate under the boundary conditions considered in the
analysis. The exact solution is
\begin{align*}
	u_1(x,y,t)
	 & =
	0.10\cos t\,\sin^2(\pi x)\sin(2\pi y),  \\
	u_2(x,y,t)
	 & =
	-0.10\cos t\,\sin(2\pi x)\sin^2(\pi y), \\
	p(x,y,t)
	 & =
	0.10\sin t\,\sin(2\pi x)\sin(2\pi y),   \\
	\phi(x,y,t)
	 & =
	0.20\cos t\,\cos(2\pi x)\cos(2\pi y),   \\
	\mu(x,y,t)
	 & =
	(8\pi^2\eps^2-1)\phi(x,y,t)+\phi(x,y,t)^3 .
\end{align*}
The velocity field is divergence-free and satisfies the no-slip boundary
condition, while $\partial_{\mathbf n}\phi = \partial_{\mathbf n}\mu = 0$
on \(\partial\Omega\). The forcing terms are obtained by substituting the
exact solution into \eqref{eq:continuous-chns}.

We take $T=\frac1{16}, \eps=0.1, \nu=0.5,$ and use uniform grids with
\(N_g=64,128,256\) cells in each coordinate direction. The time step is chosen
as \(\dt=h/2\), so that spatial and temporal errors are refined
simultaneously. The first two BDF2 levels are initialized from the exact
solution.

The errors reported in
Table~\ref{tab:manufactured-convergence}
are measured in the MAC norms of
Theorem~\ref{thm:main-convergence}, using physical restrictions rather than
projection errors:
\[
	E_\phi^n
	=
	R_c\phi(t_n)-\phi_h^n,
	\qquad
	E_\mu^n
	=
	R_c\mu(t_n)-\mu_h^n,
\]
\[
	E_u^n
	=
	R_e\bfu(t_n)-\bfu_h^n,
	\qquad
	E_p^n
	=
	R_cp(t_n)-p_h^n.
\]

\begin{table}[t]
	\caption{Convergence results for the manufactured solution test.}
	\label{tab:manufactured-convergence}
	\centering
	\scriptsize
	\begin{tabular}{@{}lcccccc@{}}
		\hline
		Error                                                    & \(N_g=64\)              & rate & \(N_g=128\)             & rate & \(N_g=256\)             & rate \\
		\hline
		\(\max_n\norm{E_\phi^n}_{H_h^1}\)                        & \(2.508\times 10^{-3}\) & --   & \(8.075\times 10^{-4}\) & 1.64 & \(2.227\times 10^{-4}\) & 1.86 \\
		\(\left(\dt\sum_n\norm{E_\mu^n}_{H_h^1}^2\right)^{1/2}\) & \(1.912\times 10^{-4}\) & --   & \(5.485\times 10^{-5}\) & 1.80 & \(1.442\times 10^{-5}\) & 1.93 \\
		\(\max_n\norm{E_u^n}_h\)                                 & \(4.903\times 10^{-5}\) & --   & \(1.267\times 10^{-5}\) & 1.95 & \(3.209\times 10^{-6}\) & 1.98 \\
		\(\left(\dt\sum_n\norm{E_p^n}_h^2\right)^{1/2}\)         & \(1.286\times 10^{-6}\) & --   & \(3.720\times 10^{-7}\) & 1.79 & \(9.850\times 10^{-8}\) & 1.92 \\
		\hline
	\end{tabular}
\end{table}

As a second test, we consider the coalescence of two diffuse droplets in a
bounded domain without external forcing. This benchmark is widely used for CHNS
solvers with no-slip velocity and no-flux phase-field boundary conditions
\cite{AndersonMcFaddenWheeler1998,KimKangLowengrub2003,Kim2012}. The initial
velocity is zero, and the phase field consists of two hyperbolic-tangent
droplets embedded in a negative background.

We use \(N_g=128\), \(\dt=2\times10^{-3}\), \(T=10\),
\(\eps=10^{-2}\), and \(\nu=10^{-2}\), together with forty pre- and
post-smoothing sweeps in each Stokes multigrid solve.
Figure~\ref{fig:coalescence-diagnostics} reports the corresponding diagnostics.
The discrete energy decays monotonically, while the maximum mass error and
discrete divergence are \(1.22\times10^{-12}\) and \(1.74\times10^{-9}\),
respectively. These results confirm the expected energy dissipation, mass
conservation, and incompressibility properties without any artificial mass
projection.

\begin{figure}[t]
	\centering
	\includegraphics[width=\textwidth]{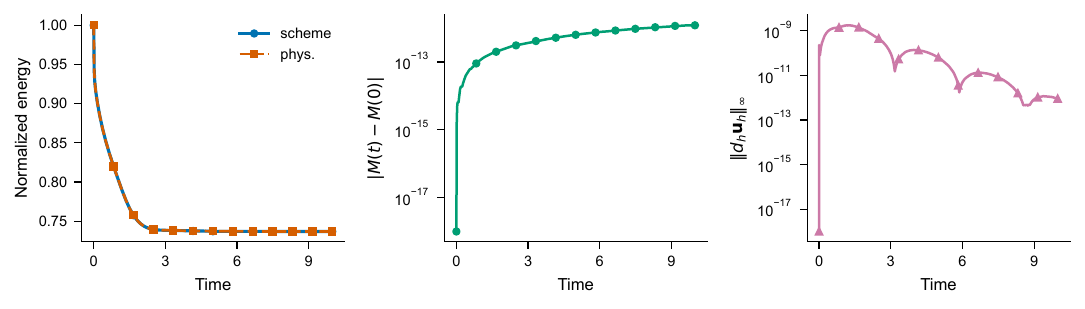}
	\caption{Global diagnostics for the droplet coalescence example. The discrete
		and physical energies are normalized by their initial values. The mass and
		divergence panels report the absolute mass error without projection and the
		maximum discrete divergence, respectively.}
	\label{fig:coalescence-diagnostics}
\end{figure}

\section{Conclusion}

We constructed and analyzed a novel implicit CS-BDF2 scheme for the CHNS system. The BDF2 bulk-energy identity
in Lemma~\ref{lem:bdf2-algebra} supplies the energy difference in the
chemical-potential equation, and discrete summation-by-parts gives the
transport--capillary cancellation. The scheme is mass conservative, uniquely
solvable, and unconditionally energy stable.

Under Assumptions~\ref{ass:start-up} and~\ref{ass:smooth-solution}, and under
the small time-step condition used in the error analysis, the method is
second-order accurate for the phase variable, chemical potential, velocity, and
pressure in the discrete norms of Theorem~\ref{thm:main-convergence}. The proof
uses summation-by-parts identities, high-norm phase estimates, nonlinear bounds,
and pressure recovery. Extension to
three-dimensional CHNS is not addressed here; the Navier--Stokes convection term
would require additional estimates.

\appendix

\section{Proof of the MAC convection estimates}\label{app:mac-convection-estimates}

\begin{proof}[Proof of Lemma~\ref{lem:mac-convection-estimates}]
	The averages in \eqref{eq:mac-averages} are bounded in the corresponding
	discrete \(\ell^p\) norms. From the pointwise convection form and skew identity in
	Lemma~\ref{lem:mac-convection-form}, for \(d_h\bfu_h=0\) one has
	\begin{align}
		|b_h(\bfu_h,\boldsymbol\xi_h,\boldsymbol\eta_h)|
		 & \le C\norm{\bfu_h}_{\ell^4}
		\norm{D_h\boldsymbol\xi_h}_h
		\norm{\boldsymbol\eta_h}_{\ell^4},
		\label{eq:app-conv-bound-l4}                        \\
		|b_h(\bfu_h,\boldsymbol\xi_h,\boldsymbol\eta_h)|
		 & \le C\norm{\bfu_h}_{\ell^\infty}
		\norm{D_h\boldsymbol\xi_h}_h
		\norm{\boldsymbol\eta_h}_h,
		\label{eq:app-conv-bound-linf}                      \\
		|b_h(\bfu_h,\boldsymbol\xi_h,\boldsymbol\eta_h)|
		 & \le C\left(\norm{\boldsymbol\xi_h}_{\ell^\infty}
		+\norm{D_h\boldsymbol\xi_h}_{\ell^\infty}\right)
		\norm{\bfu_h}_h\norm{D_h\boldsymbol\eta_h}_h .
		\label{eq:app-conv-bound-smooth}
	\end{align}
	By Lemma~\ref{lem:mac-convection-form},
	\begin{equation}\label{eq:app-skew-polarized}
		b_h(\bfu_h,\bfv_h,\boldsymbol z_h)
		=-b_h(\bfu_h,\boldsymbol z_h,\bfv_h).
	\end{equation}
	With \(\boldsymbol z_h=\bfw_h\), \eqref{eq:app-conv-bound-l4} and
	Lemma~\ref{lem:mac-discrete-estimates} with \(p=4\) give
	\begin{align*}
		|b_h(\bfu_h,\bfv_h,\bfw_h)|
		 & =|b_h(\bfu_h,\bfw_h,\bfv_h)|                       \\
		 & \le C\norm{\bfu_h}_{\ell^4}\norm{D_h\bfw_h}_h
		\norm{\bfv_h}_{\ell^4}                                \\
		 & \le C\norm{\bfu_h}_h^{1/2}\norm{D_h\bfu_h}_h^{1/2}
		\norm{\bfv_h}_h^{1/2}\norm{D_h\bfv_h}_h^{1/2}
		\norm{D_h\bfw_h}_h .
	\end{align*}

	For the dual estimates, let \(\boldsymbol z_h\) be a staggered test field. From
	\eqref{eq:app-skew-polarized} and \eqref{eq:app-conv-bound-linf},
	\begin{equation}\label{eq:app-dual-a}
		|b_h(\bfu_h,\bfv_h,\boldsymbol z_h)|
		\le C\norm{\bfu_h}_{\ell^\infty}\norm{\bfv_h}_h
		\norm{D_h\boldsymbol z_h}_h,
	\end{equation}
	which gives the first \(H_h^{-1}\) estimate. For the second dual estimate,
	\eqref{eq:app-conv-bound-smooth} with
	\(\boldsymbol\xi_h=\bfv_h\) and
	\(\boldsymbol\eta_h=\boldsymbol z_h\) gives
	\begin{align*}
		|b_h(\bfu_h,\bfv_h,\boldsymbol z_h)|
		 & \le C\left(\norm{\bfv_h}_{\ell^\infty}
		+\norm{D_h\bfv_h}_{\ell^\infty}\right)
		\norm{\bfu_h}_h\norm{D_h\boldsymbol z_h}_h,
	\end{align*}
	and hence
	\begin{equation}\label{eq:app-dual-b}
		|b_h(\bfu_h,\bfv_h,\boldsymbol z_h)|
		\le C\left(\norm{\bfv_h}_{\ell^\infty}
		+\norm{D_h\bfv_h}_{\ell^\infty}\right)
		\norm{\bfu_h}_h\norm{D_h\boldsymbol z_h}_h.
	\end{equation}
	Taking the supremum over \(\boldsymbol z_h\) gives the second dual estimate.
	For the last estimate, \eqref{eq:app-skew-polarized} and
	\eqref{eq:app-conv-bound-l4} give
	\[
		|b_h(\bfu_h,\bfv_h,\boldsymbol z_h)|
		\le C\norm{\bfu_h}_{\ell^4}\norm{\bfv_h}_{\ell^4}
		\norm{D_h\boldsymbol z_h}_h,
	\]
	and the last estimate follows by taking the supremum in
	\eqref{eq:hminusone-def}.
\end{proof}

\section{Proof of the discrete velocity projection estimate}\label{app:mac-leray}

\begin{proof}[Proof of Lemma~\ref{lem:mac-leray}]
	Let \(\pi_h\) be as in Lemma~\ref{lem:mac-leray}. Then
	\(d_hP_h\bfv_h=0\), and for
	\(\boldsymbol z_h\in Z_h\),
	\[
		(\bfv_h-P_h\bfv_h,\boldsymbol z_h)_h
		=(-D_h\pi_h,\boldsymbol z_h)_h=(\pi_h,d_h\boldsymbol z_h)_h=0.
	\]
	Thus \(P_h\) is the \(L^2_h\)-orthogonal projection onto \(Z_h\). In
	particular, \(P_h\) is self-adjoint. If \(q_h\in\mathsf C_h\), then
	\(\boldsymbol v_h=D_hq_h\) satisfies
	\[
		A_h\pi_h=d_hD_hq_h=-A_hq_h,
	\]
	so \(\pi_h=-(q_h-\overline{q_h})\), and therefore \(P_hD_hq_h=0\).

	We use the rectangular-grid Neumann estimate
	\begin{equation}\label{eq:mac-poisson-h2}
		\norm{D_hD_hr_h}_h\le C\norm{A_hr_h}_h,
		\qquad (r_h,1)_h=0,
	\end{equation}
	with \(C\) independent of \(h\). From Lemma~\ref{lem:mac-leray},
	\begin{equation}\label{eq:helmholtz-h1-stability}
		\norm{D_hP_h\bfv_h}_h
		\le \norm{D_h\bfv_h}_h+C\norm{D_hD_h\pi_h}_h
		\le C\norm{D_h\bfv_h}_h .
	\end{equation}
	Since \(P_h\) is self-adjoint and \eqref{eq:helmholtz-h1-stability} holds,
	\[
		\norm{P_h\bfv_h}_{H_h^{-1}}
		=\sup_{\boldsymbol w_h\ne0}
		\tfrac{(\bfv_h,P_h\boldsymbol w_h)_h}{\norm{D_h\boldsymbol w_h}_h}
		\le C\norm{\bfv_h}_{H_h^{-1}}.
	\]
	Also,
	\[
		\norm{A_h\bfv_h}_{H_h^{-1}}
		=\sup_{\boldsymbol w_h\ne0}
		\tfrac{(D_h\bfv_h,D_h\boldsymbol w_h)_h}{\norm{D_h\boldsymbol w_h}_h}
		\le \norm{D_h\bfv_h}_h .
	\]
	The three assertions in Lemma~\ref{lem:mac-leray} follow.
\end{proof}

\section{Proof of the weighted discrete Gronwall inequality}\label{app:gronwall-proofs}

\begin{proof}[Proof of Lemma~\ref{lem:weighted-gronwall}]
	Define \(A_m=a_m+\dt\sum_{r=1}^{m}b_r\) and
	\(S_m=\sum_{j=1}^{m}\alpha^{m-j}A_j\). Since \(a_j\le A_j\),
	the hypothesis of Lemma~\ref{lem:weighted-gronwall} gives
	\[
		A_\ell\le C_2+\dt\sum_{m=1}^{\ell-1}c_m S_m .
	\]
	The recursion \(S_\ell=A_\ell+\alpha S_{\ell-1}\) therefore implies, after
	expanding the contribution of each older index and using
	\(1+\alpha+\alpha^2+\cdots\le (1-\alpha)^{-1}\),
	\begin{equation*}
		S_\ell
		\le \tfrac{C_2}{1-\alpha}
		+\tfrac{\dt}{1-\alpha}\sum_{m=1}^{\ell-1}c_mS_m .
	\end{equation*}
	Let \(K=C_2/(1-\alpha)\) and
	\(d_m=\dt c_m/(1-\alpha)\). The preceding inequality has the explicit
	Volterra form
	\(S_\ell\le K+\sum_{m=1}^{\ell-1}d_mS_m\). Induction on \(\ell\) gives
	\[
		S_\ell\le K\prod_{m=1}^{\ell-1}(1+d_m)
		\le K\exp\left(\sum_{m=1}^{\ell-1}d_m\right)
		\le \tfrac{C_2}{1-\alpha}
		\exp\left(\tfrac{C_1}{1-\alpha}\right).
	\]
	Since \(A_\ell\le S_\ell\), this proves
	Lemma~\ref{lem:weighted-gronwall}.
\end{proof}

\raggedbottom
\section*{Acknowledgments}
Xuelong Gu's research is supported by an NSF award OIA-2242812. Qi Wang's research is
partially supported by NSF awards DMS-2038080 and OIA-2242812, a DOE award DE-SC0025229,
and an SC GEAR award.


\begin{thebibliography}{10}

\bibitem{Brower1}
{\sc G.~D. Akrivis, V.~A. Dougalis, and O.~A. Karakashian}, {\em {On fully discrete Galerkin methods of second-order temporal accuracy for the nonlinear Schr{\"o}dinger equation}}, Numer. Math., 59 (1991), pp.~31--53.

\bibitem{AndersonMcFaddenWheeler1998}
{\sc D.~M. Anderson, G.~B. McFadden, and A.~A. Wheeler}, {\em {Diffuse-interface methods in fluid mechanics}}, Annu. Rev. Fluid Mech., 30 (1998), pp.~139--165.

\bibitem{Chen-2020-CHNS}
{\sc L.~Chen and J.~Zhao}, {\em {A novel second-order linear scheme for the Cahn--Hilliard--Navier--Stokes equations}}, J. Comput. Phys., 423 (2020), p.~109782.

\bibitem{ChenWangYanZhang2019}
{\sc W.~Chen, X.~Wang, Y.~Yan, and Z.~Zhang}, {\em {A second order BDF numerical scheme with variable steps for the Cahn--Hilliard equation}}, SIAM J. Numer. Anal., 57 (2019), pp.~495--525.

\bibitem{Dahlquist1978}
{\sc G.~Dahlquist}, {\em {G-stability is equivalent to A-stability}}, BIT, 18 (1978), pp.~384--401.

\bibitem{DiegelWangWangWise2017}
{\sc A.~E. Diegel, C.~Wang, X.~Wang, and S.~M. Wise}, {\em {Convergence analysis and error estimates for a second order accurate finite element method for the Cahn--Hilliard--Navier--Stokes system}}, Numer. Math., 137 (2017), pp.~495--534.

\bibitem{Eyre1998}
{\sc D.~J. Eyre}, {\em {Unconditionally gradient stable time marching the Cahn--Hilliard equation}}, in MRS Proceedings, vol.~529, 1998, pp.~39--46.

\bibitem{Feng-2013}
{\sc X.~Feng, T.~Tang, and J.~Yang}, {\em {Stabilized Crank--Nicolson/Adams--Bashforth schemes for phase field models}}, E. Asian J. Appl. Math., 3 (2013), pp.~59--80.

\bibitem{Gong-2016-Binary}
{\sc Y.~Gong, X.~Liu, and Q.~Wang}, {\em {Fully discretized energy stable schemes for hydrodynamic equations governing two-phase viscous fluid flows}}, J. Sci. Comput., 69 (2016), pp.~921--945.

\bibitem{GongZhaoWang2018}
{\sc Y.~Gong, J.~Zhao, and Q.~Wang}, {\em {Second order fully discrete energy stable methods on staggered grids for hydrodynamic phase field models of binary viscous fluids}}, SIAM J. Sci. Comput., 40 (2018), pp.~B528--B553.

\bibitem{Gong-2018-CHNS}
{\sc Y.~Gong, J.~Zhao, X.~Yang, and Q.~Wang}, {\em {Fully discrete second-order linear schemes for hydrodynamic phase field models of binary viscous fluid flows with variable densities}}, SIAM J. Sci. Comput., 40 (2018), pp.~B138--B167.

\bibitem{EQ1}
{\sc F.~Guill{\'e}n-Gonz{\'a}lez and G.~Tierra}, {\em {On linear schemes for a Cahn--Hilliard diffuse interface model}}, J. Comput. Phys., 234 (2013), pp.~140--171.

\bibitem{GuoWangWiseYue2016}
{\sc J.~Guo, C.~Wang, S.~M. Wise, and X.~Yue}, {\em {An {$H^2$} convergence of a second-order convex-splitting, finite difference scheme for the three-dimensional {Cahn--Hilliard} equation}}, Commun. Math. Sci., 14 (2016), pp.~489--515.

\bibitem{HanWang2015}
{\sc D.~Han and X.~Wang}, {\em {A second order in time, uniquely solvable, unconditionally stable numerical scheme for Cahn--Hilliard--Navier--Stokes equation}}, J. Comput. Phys., 290 (2015), pp.~139--156.

\bibitem{HarlowWelch1965}
{\sc F.~H. Harlow and J.~E. Welch}, {\em {Numerical calculation of time-dependent viscous incompressible flow of fluid with free surface}}, Phys. Fluids, 8 (1965), pp.~2182--2189.

\bibitem{HuCheng2024}
{\sc X.~Hu and L.~Cheng}, {\em {Numerical analysis of a convex-splitting BDF2 method with variable time-steps for the Cahn--Hilliard model}}, J. Sci. Comput., 98 (2024), p.~18.

\bibitem{HuangShen2024NSCH}
{\sc F.~Huang and J.~Shen}, {\em {A class of IMEX schemes and their error analysis for the Navier--Stokes Cahn--Hilliard system}}, Ann. Math. Sci. Appl., 9 (2024), pp.~185--235.

\bibitem{Kim2012}
{\sc J.~Kim}, {\em {Phase-field models for multi-component fluid flows}}, Commun. Comput. Phys., 12 (2012), pp.~613--661.

\bibitem{KimKangLowengrub2003}
{\sc J.~Kim, K.~Kang, and J.~Lowengrub}, {\em {Conservative multigrid methods for Cahn--Hilliard fluids}}, J. Comput. Phys., 193 (2003), pp.~511--543.

\bibitem{LiShen2020NS}
{\sc X.~Li and J.~Shen}, {\em {Error analysis of the SAV-MAC scheme for the Navier--Stokes equations}}, SIAM J. Numer. Anal., 58 (2020), pp.~2465--2491.

\bibitem{LiShen2020}
\leavevmode\vrule height 2pt depth -1.6pt width 23pt, {\em {On a SAV-MAC scheme for the Cahn--Hilliard--Navier--Stokes phase-field model and its error analysis for the corresponding Cahn--Hilliard--Stokes case}}, Math. Models Methods Appl. Sci., 30 (2020), pp.~2263--2297.

\bibitem{LiShen2022MSAV}
\leavevmode\vrule height 2pt depth -1.6pt width 23pt, {\em {On fully decoupled MSAV schemes for the Cahn--Hilliard--Navier--Stokes model of two-phase incompressible flows}}, Math. Models Methods Appl. Sci., 32 (2022), pp.~457--495.

\bibitem{LiShenRui2019}
{\sc X.~Li, J.~Shen, and H.~Rui}, {\em {Energy stability and convergence of SAV block-centered finite difference method for gradient flows}}, Math. Comp., 88 (2019), pp.~2047--2068.

\bibitem{LiaoJiWangZhang2022}
{\sc H.-L. Liao, B.~Ji, L.~Wang, and Z.~Zhang}, {\em {Mesh-robustness of an energy stable BDF2 scheme with variable steps for the Cahn--Hilliard model}}, J. Sci. Comput., 92 (2022), p.~52.

\bibitem{RuiLi2017}
{\sc H.~Rui and X.~Li}, {\em {Stability and superconvergence of MAC scheme for Stokes equations on nonuniform grids}}, SIAM J. Numer. Anal., 55 (2017), pp.~1135--1158.

\bibitem{ShenXu2018}
{\sc J.~Shen and J.~Xu}, {\em {Convergence and error analysis for the scalar auxiliary variable (SAV) schemes to gradient flows}}, SIAM J. Numer. Anal., 56 (2018), pp.~2895--2912.

\bibitem{ShenYang2010}
{\sc J.~Shen and X.~Yang}, {\em {Numerical approximations of Allen--Cahn and Cahn--Hilliard equations}}, Discrete Contin. Dyn. Syst., 28 (2010), pp.~1669--1691.

\bibitem{ShenYang2015}
\leavevmode\vrule height 2pt depth -1.6pt width 23pt, {\em {Decoupled, energy stable schemes for phase-field models of two-phase incompressible flows}}, SIAM J. Numer. Anal., 53 (2015), pp.~279--296.

\bibitem{Brower2}
{\sc P.~Wang and C.~Huang}, {\em {An energy conservative difference scheme for the nonlinear fractional Schr{\"o}dinger equations}}, J. Comput. Phys., 293 (2015), pp.~238--251.

\bibitem{WeiserWheeler1988}
{\sc A.~Weiser and M.~F. Wheeler}, {\em {On convergence of block-centered finite differences for elliptic problems}}, SIAM J. Numer. Anal., 25 (1988), pp.~351--375.

\bibitem{XueZhaiZhao2024}
{\sc Z.~Xue, S.~Zhai, and X.~Zhao}, {\em {Energy dissipation and evolutions of the nonlocal Cahn--Hilliard model and space fractional variants using efficient variable-step BDF2 method}}, J. Comput. Phys., 510 (2024), p.~113071.

\bibitem{YanChenWangWise2018}
{\sc Y.~Yan, W.~Chen, C.~Wang, and S.~M. Wise}, {\em {A second-order energy stable BDF numerical scheme for the Cahn--Hilliard equation}}, Commun. Comput. Phys., 23 (2018), pp.~572--602.

\bibitem{Zhao-2021-CHNS}
{\sc J.~Zhao and D.~Han}, {\em {Second-order decoupled energy-stable schemes for Cahn--Hilliard--Navier--Stokes equations}}, J. Comput. Phys., 443 (2021), p.~110536.

\end{thebibliography}
\end{document}